\documentclass[reqno,twoside]{amsart}
\usepackage{amssymb,amsmath,latexsym,hhline,graphics}
\usepackage{amsfonts,layout}
\usepackage{fullpage}
\usepackage{microtype}
\DisableLigatures{encoding = *, family = * }

\usepackage[colorlinks=true,citecolor=red,linkcolor=blue,pdfpagetransition=Blinds]{hyperref}

\title{Existence and cost of boundary controls for a degenerate/singular parabolic equation}

\author{U. Biccari}  
\address{Chair of Computational Mathematics, Fundaci\'on Deusto, Avda de las Universidades 24, 48007 Bilbao, Basque Country, Spain.}
\address{Facultad de Ingenier\'ia, Universidad de Deusto, Avda de las Universidades 24, 48007 Bilbao, Basque Country, Spain.}
\email{umberto.biccari@deusto.es, u.biccari@gmail.com}

\author{V. Hern\'andez-Santamar\'ia} 
\address{Institut de math\'ematiques de Toulouse, UMR5219; Universit\'e de Toulouse; CNRS, UPS IMT, F-31062 Toulouse Cedex 9, France}
\email{victor.santamaria@math.univ-toulouse.fr}

\author{J. Vancostenoble} 
\address{Institut de math\'ematiques de Toulouse, UMR5219; Universit\'e de Toulouse; CNRS, UPS IMT, F-31062 Toulouse Cedex 9, France}
\email{vancoste@math.univ-toulouse.fr}

\date{\today}
\thanks{This project has received funding from the European Research Council (ERC) under the European Union’s Horizon 2020 research and innovation programme (grant agreement NO. 694126-DyCon). The work of U. B. was partially supported by the Grant MTM2017-92996-C2-1-R COSNET of MINECO (Spain) 	and by the Air Force Office of Scientific Research (AFOSR) under Award NO. FA9550-18-1-0242.}

\newtheorem{Theorem}{Theorem}[section]
\newtheorem{Definition}{Definition}[section]
\newtheorem{Proposition}{Proposition}[section]
\newtheorem{Lemma}{Lemma}[section]

\numberwithin{equation}{section}

\newtheorem{Remark}{Remark}[section]

\newtheorem{Hyp.}{Hyp.}[section]

\newcommand{\NN}{\mathbb{N}}

\newcommand{\jnk}{j_{\nu(\alpha,\mu),k}}

\newcommand{\norm}[2]{\left\|#1\right\|_{#2}}

\newcommand{\C}{\mathcal C}

\begin{document}

\begin{abstract}

In this paper, we consider the following degenerate/singular parabolic equation
\begin{align*}
	u_t -(x^\alpha u_{x})_x - \frac{\mu}{x^{2-\alpha}} u =0, \qquad x\in (0,1), \ t \in (0,T),
\end{align*}
where $0\leq \alpha <1$ and $\mu\leq (1-\alpha)^2/4$ are two real parameters. 

We prove the boundary null controllability by means of a $H^1(0,T)$ control acting either at $x=1$ or at the point of degeneracy and singularity $x=0$. Besides we give sharp estimates of the cost of controllability in both cases in terms of the parameters $\alpha$ and $\mu$. The proofs are based on the classical moment method by Fattorini and Russell and on recent results on biorthogonal sequences.

\end{abstract}
\maketitle

\section{Introduction and main results}

The aim of this paper is to prove boundary null controllability for some degenerate/singular parabolic equation in $1-D$ and establish sharp estimates of the control cost. More precisely, we focus on  the following degenerate/singular operator:
\begin{align*}
	\mathbf P_{\alpha,\mu} u : = u_t -(x^\alpha u_x)_x - \frac{\mu}{x^{2-\alpha}} u, \qquad x\in (0,1).
\end{align*}
Observe that, when $\mu=0$, this operator is purely degenerate: 
\begin{align*}
	{\mathbf P_{\alpha}} u := {\mathbf P_{\alpha,0}} u =u_t -(x^\alpha u_x)_x , \qquad x\in (0,1),
\end{align*}
whereas, when $\alpha=0$, it becomes purely singular with a singularity that takes the form of an inverse square potential:
\begin{align*}
	{\mathbf P_{\mu}}u := {\mathbf P_{0,\mu}} u =  u_t - u_{xx} - \frac{\mu}{x^2} u, \qquad x\in (0,1).
\end{align*}

Null controllability properties by means of {\bf a locally distributed control} for such operators have been investigated in various papers.  We refer the reader to the following pioneering contributions:

\begin{itemize}
	\item Concerning  the degenerate operator ${\mathbf P_{\alpha}}$, the first complete result was obtained in \cite{sicon2008} and shows that null controllability holds true if and only if $0 \leq \alpha <2$. One distinguishes  here the two cases $0 < \alpha < 1$ and $1 \leq \alpha <2$ for well-posedness reasons: in the natural functional setting associated to the {\it weakly} degenerate operator, that is when  $0 < \alpha < 1$, the trace at $x=0$ exists. So one can consider a Dirichlet condition at $x=0$. On the contrary, the trace does not exist when $\alpha \geq 1$. Here  the Dirichlet boundary condition needs to be changed by some Neumann-kind one. We also refer to \cite{fatiha, non-div, memoire,Martinez} for various related results.

	\item Concerning the inverse square singular operator ${\mathbf P_{\mu}}$, the first study was made in \cite{Va-Zu-JFA} and complemented in \cite{Ervedoza}. In these references, it was shown that null controllability holds true if and only if $\mu \leq \mu^\star$ where  $\mu^\star = 1/4$ is the constant appearing in the well-known Hardy inequality
\begin{align*}
	 \frac 14\int_0^1\frac{z^2}{x^2}\,dx \leq  \int_0^1  z_x^2\, dx .
 \end{align*}
Also here one distinguishes the two cases $\mu < \mu^\star$ and $\mu=\mu^\star$ again for well-posedness reasons: the natural functional space in the critical case $\mu=\mu^\star$ slighly differs from the one in the general sub-critical case. Let us refer to \cite{biccari,cazacu} for other similar situations. 

	\item Finally, the mixed degenerate/singular operator  ${\mathbf P_{\alpha, \mu}}$ was studied in \cite{vanco1}. Null controllability here holds true if and only if 
\begin{align*}
	0 \leq \alpha <2 \quad \text{ and } \quad \mu \leq \mu(\alpha)
\end{align*}
where 
\begin{align*}
	\mu(\alpha): = \frac{(1-\alpha)^2}{4}
\end{align*}
is the constant appearing in the \textit{generalized Hardy inequality}
\begin{equation}\label{generalized-hardy}
	\frac{(1-\alpha)^2}{4} \int_0^1 \frac{z^2}{x^{2-\alpha}}\, dx \leq  \int_0^1 x^\alpha z_x^2\,dx .
\end{equation}
See also \cite{Fo-Sa, Maniar} for other works on this theme. 
\end{itemize}

Dealing with locally distributed controls, all these mentioned contributions are mainly based on a Carleman approach, suitably adapted for taking into account the degeneracy/singularity in the equation. 

In the present paper, we turn to {\bf the case of a boundary control}. To our knowledge, the question has never been studied for the degenerate/singular operator ${\mathbf P_{\alpha, \mu}}$. However, some recent works studied this problem for the purely degenerate operator ${\mathbf P_{\alpha}}$ and next for the purely singular one ${\mathbf P_{\mu}}$. In more detail:
\begin{itemize}
	\item In the case of ${\mathbf P_{\alpha}}$, the first result of boundary controllability was obtained in \cite{Gueye} for a control acting at $x=0$ and in the case of a weak degeneracy $0\leq \alpha <1$. It has been complemented in \cite{cost-weak} where sharp estimates of cost of the control have been obtained. Next,  in \cite{CMV-the-cost-strong}, the case of a strong degeneracy ($1\leq\alpha <2$) with a control acting at $x=1$ has been studied, analyzing again both the existence and the cost of the control. 

	\item In the case of ${\mathbf P_{\mu}}$, the boundary controllability from $x=1$ has been studied in \cite{MV-cost-singular} whereas the case of a control at $x=0$ is treated in \cite{biccari2}.  
\end{itemize}

When addressing the boundary controlability problem, the approach by Carleman estimates presents some difficulties. Indeed, the specific weight functions introduced in \cite{sicon2008,vanco1,Va-Zu-JFA} to deal with the degeneracy and/or the singularity do not provide suitable boundary terms. For this reson, the approach of the aforementioned papers is based on decomposition in series and the well-known moment method. This is the methodology that we will employ also in the present work. 

The rest of the paper is organized as follows: in Section \ref{results_sec}, we formultate precisely the problems we are going to study and we present our main theorems. In Section \ref{preliminary_sec}, we introduce some preliminary results on the spectral properties of the operator ${\mathbf P}_{\alpha,\mu}$ which will then be at the basis of our proofs. Moreover, we briefly describe the main procedure of the moment method. Section \ref{well_pos_sec} is devoted to the well-posedness of our problems. In Section \ref{sec-pf-control_thm1} and \ref{sec-pf-control_thm0}, we present the proof of our main results, Theorem \ref{control_thm1} and \ref{control_thm0}. Finally, in Section \ref{open_pb_sec} we give some conclusive remarks and open problems.

\section{Problem formulation and main results}\label{results_sec}

\subsection{Description of the controllability problem}
Let us describe more precisely the controllability problems we are interested in. First of all, we  focus in this work on the case of a {\it weak degeneracy}, that is, $0\leq\alpha<1$. The case of a strong degeneracy ($1\leq\alpha<2$) requires a change of boundary condition and it will be treated in a future work. So throughout the paper, we assume that the parameters $\alpha$ and $\mu$ satisfy the following assumption: 
\begin{equation}\label{hyp-parameters}
	0 \leq \alpha <1  \quad \text{ and } \quad \mu \leq \mu(\alpha)=\frac{(1-\alpha)^2}{4}.
\end{equation}

Moreover we will consider boundary controls acting either at $x=1$ or at the degeneracy/singularity point $x=0$.

\subsubsection{Control acting away from the degenerate/singular point} We will first study the case of a boundary control acting at $x=1$ (that is, away from the degenerate and singular point): let $u_0 \in L^2(0,1)$, $T>0$ and consider 
\begin{equation}\label{pb-x-1}
	\begin{cases}
		\displaystyle u_t -(x^\alpha u_x)_x - \frac{\mu}{x^{2-\alpha}} u =0,  & (x,t)\in(0,1)\times(0,T):=Q
		\\
		u(0,t)=0, & t\in(0,T) 
		\\
		u(1,t)=H(t), & t\in(0,T) 
		\\
		u(x,0)=u_0(x), & x\in(0,1).
	\end{cases}
\end{equation}

Here $H$ represents some control term that aims to steer the solution to zero at time $T$. Our first goal is to establish the existence of such control (which could in this case be deduced from the result of controllability by a locally distributed control via the method of extension of the domain). Moreover, in this work we are also interested in providing sharp estimates of the cost of such control, in dependence of the parameters $\alpha$ and $\mu$ entering in our problem. 

\subsubsection{Control acting at the degenerate/singular point} 
Next we will turn to the case of a control acting at $x=0$ (that is on the point of degeneracy and singularity). In this case,  even the existence of a control is new since it cannot be deduced from the result of controllability by a locally distributed control. Moreover, as previously, we also aim at estimating precisely the cost of the control. The problem we consider here is: 
\begin{equation}\label{pb-x-0}
	\begin{cases}
		\displaystyle u_t -(x^\alpha u_x)_x - \frac{\mu}{x^{2-\alpha}} u =0,  & (x,t)\in Q
		\\
		(x^{-\gamma}u)(0,t)=H(t), & t\in(0,T)
		\\
		u(1,t)=0, & t\in(0,T)
		\\
		u(x,0)=u_0(x), & x\in(0,1).
	\end{cases}
\end{equation}

Due to the presence of the singularity at $x=0$, it is not possible to impose a standard non homogeneous Dirichlet boundary condition. For this reason, as in \cite{biccari2}, we use the above weighted Dirichlet condition where the coefficient $\gamma$ is defined by
\begin{equation}\label{def-gamma}
	\gamma=\gamma(\alpha,\mu):= \frac{1-\alpha}{2}-\frac 12\sqrt{(1-\alpha)^2-4\mu} =\sqrt{\mu(\alpha)}- \sqrt{\mu(\alpha)-\mu}.
\end{equation}
Notice that we have 
\begin{align*}
	\gamma(\alpha,0)=0 \quad \textrm{ and } \quad \gamma(0,\mu) = \frac 12\Big(1-\sqrt{1-4\mu}\,\Big), 
\end{align*}
consistently with \cite{biccari2,cost-weak}.

\subsection{Main results}

We present here the main results of the paper. To this end, we firs need to introduce the following notion of controllability cost.

For any $T>0$, $0\leq \alpha <1$, $\mu \leq \mu(\alpha)$ and any initial datum $u_0 \in L^2(0,1)$, we introduce the set of admissible controls:
\begin{align*}
	\mathcal U_\text{ad}(\alpha, \mu, T, u_0) := \Big\{ H \in H^1(0,T)  \;\Big|\; u^{(H)} (T)=0 \Big\},
\end{align*}
where $u^{(H)}$ denotes the solution of \eqref{pb-x-1} or of \eqref{pb-x-0} corresponding to the control $H$. Then we consider the controllability cost for any $u_0 \in L^2(0,1)$
\begin{align*}
	\C^{H^1} (\alpha, \mu, T, u_0) := \mathop{\text{inf}}_{H \in \mathcal U_\text{ad}(\alpha, \mu, T, u_0)} \norm{H}{H^1(0,T)},
\end{align*}
which is the minimal energy needed to drive the initial datum $u_0$ to zero.  Finally, we define the global notion of controllability cost:
\begin{align*}
	\C^{H^1}_{bd-ctr} (\alpha, \mu, T) := \mathop{\text{sup}}_{\norm{u_0}{L^2(0,1)}=1}
\C^{H^1} (\alpha, \mu, T, u_0).
\end{align*}

\subsubsection{Results for a control acting at $x=1$}
Our first main result, concerning the existence of a control for equation \eqref{pb-x-1}, will be the following.
\begin{Theorem}\label{control_thm1}
Let $0\leq\alpha<1$ and $\mu\leq\mu(\alpha)$. Given any $T>0$ and $u_0\in L^2(0,1)$, the following assertions hold :
\begin{itemize}
	\item[{\bf (i)}] \textbf{Existence of a control.}
	There exists a control function $H\in H^1(0,T)$ such that the solution of \eqref{pb-x-1} satisfies $u(x,T)=0$.

	\item[{\bf (ii)}] \textbf{Upper bound of the cost.}
	There exists a constant $\C_u>0$, independent of $\alpha$, $\mu$ and $T$, such that the cost of null controllability for \eqref{pb-x-1} satisfies
	\begin{align*}
 		\C^{H^1}_{bd-ctr}(\alpha, \mu,T)\leq \C_u e^{\frac{\C_u}{T}}\left[1+\sqrt{\mu(\alpha)-\mu}\right]e^{- \C_u\left[1+\sqrt{\mu(\alpha)-\mu}\right]^2 T}.
	\end{align*}
	\item[{\bf (iii)}] \textbf{Lower bound of the cost.}
	There exists a constant $\C_u>0$, independent of $\alpha$, $\mu$ and $T$, such that the cost of null controllability for \eqref{pb-x-1} satisfies: 
	\begin{itemize}
		\item[$\bullet$] in the case
		\begin{align*}
			\nu(\alpha, \mu) \in \left[0, \frac{1}{2}\right], \quad \text{ that is, } \mu \in \left[ \frac{\alpha}{16} (3\alpha-4), \mu(\alpha)\right],
		\end{align*}
		then
		\begin{align*}
			\C_{ctr-bd} \geq \C_u  e^{\frac{\C_u}{ T}}  e^{- \C_u \left[1+  \sqrt{\mu(\alpha)-\mu}\right]^2 T};
		\end{align*}
		\item[$\bullet$] in the case 
		\begin{align*} 
			\nu(\alpha, \mu) \in \bigg[\frac{1}{2}, +\infty \bigg),  \quad \text{ that is, } \mu \in \bigg( - \infty, \frac{\alpha}{16} (3\alpha-4) \bigg],
		\end{align*}
		then 
		\begin{align*}
			\C_{ctr-bd} \geq 
			\C_u\,e^{\frac{\C_u}{T}}e^{-\C_u \left[1+ \sqrt{\mu(\alpha)-\mu} \right]^2 T}\, e^{-\C_u \left[\sqrt{\mu(\alpha)-\mu} \right]^{4/3} \left(\ln\left[\sqrt{\mu(\alpha)-\mu}\right] + \ln \frac{1}{T} \right)}.
		\end{align*}	
	\end{itemize}
\end{itemize}
\end{Theorem}
The proof of Theorem \ref{control_thm1} will be given in Section \ref{sec-pf-control_thm1}.

\subsubsection{Results for a control acting at $x=0$}

The second main result of our work concerns the existence of a control for equation \eqref{pb-x-0}, and it reads as follows.

\begin{Theorem}\label{control_thm0}
Let $0\leq\alpha<1$ and $\mu<\mu(\alpha)$. Given any $T>0$ and $u_0\in L^2(0,1)$, the following assertions hold:
\begin{itemize}
	\item[{\bf (i)}] \textbf{Existence of a control.} There exists a control function $H\in H^1(0,T)$ such that the solution of \eqref{pb-x-0} satisfies $u(x,T)=0$.
	\item[{\bf (ii)}] \textbf{Upper bound of the cost.} There exists a constant $\C_u>0$, independent of $\alpha$, $\mu$ and $T$, such that the cost of null controllability for \eqref{pb-x-0} satisfies
	\begin{align*}
		\C _{ctr-bd} \leq \C_u\frac{\Gamma(1+\nu(\alpha,\mu))}{\sqrt{\mu(\alpha)}+\sqrt{\mu(\alpha)-\mu}} e^{\frac{\C_u}{T}} \left[1+  \sqrt{\mu(\alpha)-\mu}\right] e^{- \C_u \left( 1+\sqrt{\mu(\alpha)-\mu} \right)^2 T}  .
	\end{align*}
	\item[{\bf (iii)}] \textbf{Lower bound of the cost.} There exists a constant $\C_u>0$, independent of $\alpha$, $\mu$ and $T$, such that the cost of null controllability for \eqref{pb-x-0} satisfies: 
	\begin{itemize}
		\item[$\bullet$] in the case
			\begin{align*} 
				\nu(\alpha, \mu) \in \left[0, \frac{1}{2}\right], \quad \text{ that is, } \mu \in \left[ \frac{\alpha}{16} (3\alpha-4), \mu(\alpha)\right],
			\end{align*}
			then
			\begin{align*}
				\C_{ctr-bd} \geq \frac{\C_u}{\sqrt{\mu(\alpha)}+\sqrt{\mu(\alpha)-\mu}}\frac{1}{T^4}e^{-\C_u(1-\alpha)^2T}e^{\frac{\C_u}{T}};
			\end{align*}
	
		\item[$\bullet$] in the case 
			\begin{align*}
				\nu(\alpha, \mu) \in \bigg[\frac{1}{2}, +\infty \bigg), 	\quad \text{ that is, } \mu 	\in \bigg( - \infty, \frac{\alpha}{16} (3\alpha-4) \bigg],
			\end{align*}	
			then 
			\begin{align*}
				\C_{ctr-bd} &\geq \frac{\C_u}{\sqrt{\mu(\alpha)}+\sqrt{\mu(\alpha)-\mu}}\,e^{\frac{\C_u}{T}}e^{-\C_u \left[1+ \sqrt{\mu(\alpha)-\mu} \right]^2 T}\, e^{-\C_u \left[\sqrt{\mu(\alpha)-\mu} \right]^{4/3} \left(\ln\left[\sqrt{\mu(\alpha)-\mu}\right] + \ln \frac{1}{T} \right)}.
			\end{align*}
		\end{itemize}
\end{itemize}
\end{Theorem}

The proof of Theorem \ref{control_thm0} will be given in Section \ref{sec-pf-control_thm0}.

\section{Preliminary results}\label{preliminary_sec}
The strategy for proving Theorems \ref{control_thm1} and \ref{control_thm0} is based on the moment method (see \cite{FR1,FR2}) and requires the study of the associated Sturm-Liouville problem: one needs the expressions of the eigenvalues and eigenfunctions together with suitable estimates on the eigenvalues. We summarize here all these preliminary results that are useful to transform the controllability problems into moment problems and solve these questions. 

\subsection{Spectral properties of the operator ${\mathbf P}_{\alpha,\mu}$}
In order to transform the question of null controllability into a moment problem, we first study the eigenvalue problem associated to the degenerate/singular operator ${\mathbf P}_{\alpha,\mu}$:
\begin{equation}\label{eigenproblem}
	\begin{cases}
		\displaystyle {\mathbf P}_{\alpha,\mu} \phi = -  (x^\alpha \phi')'(x)  -\frac{\mu}{x^{2-\alpha}}\phi(x)  =\lambda \phi(x), & x\in (0,1)
		\\
		\phi (0)= 0 = \phi(1).   
	\end{cases}
\end{equation}
We prove:
\begin{Proposition}\label{prop-eigenelements}
Assume $0 \leq \alpha <1$ and $\mu \leq \mu(\alpha)$ and define
\begin{equation}\label{nu} 
	\nu(\alpha, \mu):= \frac{2}{2-\alpha} \, \sqrt{ \left(\frac{1-\alpha}{2} \right)^2 - \mu} = \frac{2}{2-\alpha} \, \sqrt{ \mu(\alpha) - \mu}.
\end{equation} 
For any $\nu \geq 0$, we denote by $J_{\nu }$ the Bessel function of first kind of order $ \nu$  and we denote
\begin{align*}
	0< j_{\nu,1} < j_{\nu,2} < \dots < \ j_{\nu,k} < \dots \to +\infty \text{ as } k \to +\infty
\end{align*}
the sequence of positive zeros of  $J_{\nu }$. Then the admissible eigenvalues $\lambda$ for problem \eqref{eigenproblem} are 
\begin{equation}\label{eigenv}
	\forall k \geq 1, \qquad \lambda_{\alpha, \mu,k} = \left(\frac{2-\alpha}{2} \right)^2 ( j_{\nu(\alpha, \mu),k})^2
\end{equation} 
and the corresponding (normalized) eigenfunctions are
\begin{equation*}
	\forall k \geq 1, \qquad  \Phi_{\alpha,\mu,k} (x)= \frac{\sqrt{2-\alpha}}{\vert J'_{\nu(\alpha, \mu)} (j_{\nu(\alpha, \mu) ,k}) \vert } x^{\frac{1-\alpha}{2}} J_{\nu(\alpha, \mu) } \left( j_{\nu(\alpha, \mu)  ,k} x^{\frac{2-\alpha}{2}} \right).
\end{equation*}
Moreover the family $( \Phi_{\alpha, \mu,k})_{k\geq 1}$ forms an orthonormal basis of $L^2(0,1)$. 
\end{Proposition}

\begin{proof}
Let us first prove that that any admissible eigenvalue $\lambda$ satisfies $\lambda>0$. Let $\phi$ be an eigenfunction. Multiplying the equation by $\phi$ and integrating by parts over $(0,1)$, we get
\begin{align*}
	\int_0^1 \left( x^\alpha \phi_x^2 -\frac{\mu}{x^{2-\alpha}} \phi^2  \right)\, dx = \lambda \int_0^1 \phi ^2\, dx.
\end{align*}
Using $\mu \leq \mu(\alpha)$ and the generalized Hardy inequality \eqref{generalized-hardy}, we have 
\begin{align*}
	\lambda \int_0^1 \phi ^2\, dx \geq \int_0^1 \left( x^\alpha \phi_x^2 -\frac{\mu(\alpha)}{x^{2-\alpha}} \phi^2  \right)\, dx \geq 0.
\end{align*}
It follows that $\lambda\geq 0$ since $\phi \not \equiv 0$. Assume now that $\lambda=0$. Then 
\begin{align*}
	\int_0^1 \left( x^\alpha \phi_x^2 -\frac{\mu}{x^{2-\alpha}} \phi^2  \right)\, dx = \lambda \int_0^1 \phi ^2\, dx=0.
\end{align*}

This implies that $\phi \equiv 0$ since the left hand side of the above relation defines a norm on $H^{1,\mu}_{\alpha, 0}(0,1)$. (It is a consequence of \eqref{generalized-hardy} when $\mu < \mu(\alpha)$ and of \cite[Theorem 2.2]{vanco1} when $\mu=\mu(\alpha)$). Since $\phi \not \equiv 0$, it follows that $\lambda >0$. 

In view of the above discussion, in what follows we will always assume $\lambda >0$. Now, using the changes of variables
\begin{align*}
	\phi(x)= x^{\frac{1-\alpha}{2}} \psi \left(\frac{2}{2-\alpha} \sqrt{\lambda} x^{\frac{2-\alpha}{2}}\right) \quad \text{ and } \quad y=\frac{2}{2-\alpha}\sqrt{\lambda}x^{\frac{2-\alpha}{2}},
\end{align*}
one can easily see that $\phi $ satisfies \eqref{eigenproblem} if and only if $\psi$ is solution of 
\begin{equation*}
	\begin{cases}
		\displaystyle y^2  \psi''(y)  +y \psi ' (y) + \big(y^2 -\nu(\alpha,\mu)^2\big) \psi (y) =0,  &  \displaystyle y \in \left(0,\frac{2\sqrt{\lambda}}{2-\alpha} \right)
		\\
		\displaystyle\psi (0)=0 = \psi \left(\frac{2\sqrt{\lambda}}{2-\alpha} \right).
	\end{cases}
\end{equation*}

Hence $\psi$ is a solution of the Bessel equation of order $\nu(\alpha,\mu)$. A fundamental system of solutions of the above Bessel equation is given by $\{ J_{\nu(\alpha, \mu)} , Y_{\nu(\alpha, \mu)} \}$, where  $J_{\nu(\alpha, \mu)}$ and $Y_{\nu(\alpha, \mu)} $ are the Bessel's functions of order $\nu(\alpha,\mu)$, respectively of the first kind and of second kind. So $\Psi$ takes the form: 
\begin{align*}
	\forall y \in \left(0,\frac{2\sqrt{\lambda}}{2-\alpha} \right), \quad \Psi (y) = \C J_{\nu(\alpha,\mu)} (y) + \C' Y_{\nu(\alpha,\mu)} (y) ,
\end{align*}
for some  $\C, \C' \in \mathbb R$. It is known that $J_{\nu(\alpha, \mu)}(0)=0$ and $Y_{\nu(\alpha, \mu)} (0)=- \infty$ (see \cite[sections 5.3 and 5.4]{Lebedev}). In order to satisfy the boundary condition at $x=0$, it  follows that $\C'=0$. Thus
\begin{align*}
	\forall y \in \left(0,\frac{2\sqrt{\lambda}}{2-\alpha} \right), \quad \Psi (y) = \C J_{\nu(\alpha,\mu)} (y), 
\end{align*}
with  $\C\not =0$. Then the other boundary condition implies that 
\begin{align*}	
	J_{\nu(\alpha,\mu)} \left( \frac{2\sqrt{\lambda}}{2-\alpha} \right) =0.
\end{align*}
So one has 
\begin{align*}
	\frac{2\sqrt{\lambda}}{2-\alpha}= j_{\nu(\alpha, \mu),k},
\end{align*}
for some $k \in \mathbb N^\star$. Therefore the set of admissible eigenvalues is given by
\begin{align*}
	\lambda_{\nu(\alpha, \mu),k} = \left( \frac{2-\alpha}{2} \right)^2 j_{\nu(\alpha, \mu),k}^2, \qquad k \in \mathbb N ^* .
\end{align*}
As for the eigenfunctions, they take the form 
\begin{align*}
	\forall k \geq 1, \qquad  \Phi_{\alpha,\mu,k} (x)= \C_k x^{\frac{1-\alpha}{2}} J_{\nu(\alpha, \mu) } \left( j_{\nu(\alpha, \mu)  ,k} x^{\frac{2-\alpha}{2}} \right).
\end{align*}

It remains to show that $(\Phi_k)_{k\geq 1}$ forms an orthogonal family in $L^2(0,1)$ and choose $\C_k$ so that it becomes normalized. 
For any $n, m \geq 1$, let us compute 
\begin{align*}
	\int _0 ^1  \Phi _{\nu(\alpha,\mu), n} (x)  &\Phi _{\nu(\alpha,\mu), m}  (x) \, dx 
	\\
	&= \C_n\C_m \int _0 ^1 x^{1-\alpha} J_{\nu(\alpha, \mu) } \left( j_{\nu(\alpha, \mu),n} x^{\frac{2-\alpha}{2}} \right) J_{\nu(\alpha, \mu) } 
\left( j_{\nu(\alpha, \mu)  ,m} x^{\frac{2-\alpha}{2}} \right) \,dx
	\\
	&= \frac{2\C_n\C_m}{2-\alpha}  \int _0 ^1 y J_{\nu(\alpha, \mu) } \left( j_{\nu(\alpha, \mu)  ,n} y \right) J_{\nu(\alpha, \mu) } \left( j_{\nu(\alpha, \mu)  ,m} y \right) \,dx
	\\
	&=\frac{\C_n\C_m}{2-\alpha} \delta_{nm} [J_{\nu(\alpha,\mu)+1} (j_{\nu(\alpha,\mu) ,n} ) ]^2, 
\end{align*}
where we used the orthogonality property of Bessel's functions (see \cite[section 5.14]{Lebedev}).  Moreover, Bessel's functions satisfy the identity (see \cite[p. 45, equation (4)]{Watson}):
\begin{align*}
	x J_\nu '(x) - \nu J_\nu (x) = -x J_{\nu+1} (x), 
\end{align*}
yielding $J_{\nu+1}(j_{\nu,n})= - J_\nu'(j_{\nu,n})$. It follows that 
\begin{align*}
	\int _0 ^1  \Phi _{\nu(\alpha,\mu), n} (x)  \Phi _{\nu(\alpha,\mu), m}  (x) \, dx = \frac{\C_n\C_m}{2-\alpha} \delta_{nm} [J_{\nu(\alpha, \mu)}' (j_{\nu(\alpha, \mu) ,n} ) ]^2.
\end{align*}
Finally,  choosing 
\begin{align*}
	\C_k= \frac{\sqrt{2-\alpha}}{|J_{\nu(\alpha, \mu)}' (j_{\nu(\alpha, \mu) ,k} )|},
\end{align*}
the family $(\Phi_k)_{k\geq 1}$ is orthonormal in $L^2(0,1)$. 
\end{proof}

Next we give some estimates on the eigenvalues  that will be useful in the analysis of the problem. Referring to \cite[Section 15.53]{Watson}, we can give the following asymptotic expansion of the zeros of the Bessel function $J_{\nu}$, for any fixed $\nu\geq 0$:
\begin{align*}
	j_{\nu,k}=\left(k+\frac{\nu}{2}-\frac{1}{4}\right)\pi-\frac{4\nu^2-1}{8\left(k+\frac{\nu}{2}-\frac{1}{4}\right)\pi}+O\left(\frac{1}{k^3}\right),\;\textrm{ as } k\to +\infty.
\end{align*}

Moreover, in what follows we will also need the following bounds on $j_{\nu,k}$, which are  provided in \cite[Lemma 1]{Lorch}
\begin{equation}\label{bessel_zero_bound}
	\begin{cases}
		\displaystyle\forall \nu\in \left[0,\frac{1}{2}\right], & \forall k\geq 1, \quad \displaystyle\pi\left(k+\frac{\nu}{2}-\frac{1}{4}\right)\leq j_{\nu,k}\leq\pi\left(k+\frac{\nu}{4}-\frac{1}{8}\right), 
		\\[10pt]
		\displaystyle\forall \nu\in \left[\frac{1}{2},+\infty\right), & \forall k\geq 1, \quad  \displaystyle\pi\left(k+\frac{\nu}{4}-\frac{1}{8}\right)\leq j_{\nu,k}\leq\pi\left(k+\frac{\nu}{2}-\frac{1}{4}\right).
	\end{cases}
\end{equation}

The inequalities above become exact when $\nu =1/2$ (which corresponds, according to \eqref{nu}, to $\alpha=\mu=0$). We also recall the following result, whose proof is classical and can be found in \cite[Proposition 7.8]{Kom-Lor}.
\begin{Lemma}\label{eigen_gap_lemma}
Let $(j_{\nu,k})_{k\geq 1}$ be the sequence of positive zeros of the Bessel function $J_{\nu}$. Then the following holds:
\begin{itemize}
	\item The difference sequence $(j_{\nu,k+1}-j_{\nu,k})_k$ converges to $\pi$ as $k\to +\infty$.
		
	\item The sequence $(j_{\nu,k+1}-j_{\nu,k})_k$ is strictly decreasing if $|\nu| > 1/2$, strictly increasing if $|\nu| < 1/2$, and constant if $|\nu| = 1/2$.
\end{itemize}
\end{Lemma}

In addition, using the bounds on the zeros of Bessel functions that we just presented, we can provide upper and lower bounds, uniform with respect to $k$, for the difference $\sqrt{\lambda_{\alpha,\mu,k+1}}-\sqrt{\lambda_{\alpha,\mu,k}}$ between the square roots of two successive eigenvalues of our original problem. These bounds will be crucial in the proof of the controllability result and in the estimation of the controllability cost. In more detail, we have the following result.

\begin{Lemma}\label{gap_lemma}
We have the following bounds for the difference $\sqrt{\lambda_{\alpha,\mu,k+1}}-\sqrt{\lambda_{\alpha,\mu,k}}$:
\begin{itemize}
	\item[(i)] When $\nu(\alpha,\mu)\in\left[0,\frac{1}{2}\right)$ that is when $\displaystyle{\mu \in \left( \frac{\alpha}{16} (3\alpha-4), \mu(\alpha)\right] },$ then 
	\begin{equation}\label{eigen-gap1}
		\forall k \geq 1, \qquad \frac{7\pi}{16}(2-\alpha)\leq\sqrt{\lambda_{\alpha,\mu,k+1}}-\sqrt{\lambda_{\alpha,\mu,k}}\leq \frac{(2-\alpha)}{2}\pi.
	\end{equation}
	\item[(ii)] When $\nu(\alpha,\mu)\in\left[\frac{1}{2},+\infty\right)$ that is when $ \displaystyle{ \mu \in \left( -\infty,  \frac{\alpha}{16} (3\alpha-4) \right]},$ then
	\begin{equation}\label{eigen-gap2}
		\forall k \geq 1, \qquad \frac{\pi}{2}(2-\alpha)\leq \sqrt{\lambda_{\alpha,\mu,k+1}}-\sqrt{\lambda_{\alpha,\mu,k}}\leq \frac{(2-\alpha)}{2}
	\Big( j_{\nu(\alpha,\mu),2} - j_{\nu(\alpha,\mu),1} \Big). 
	\end{equation}
\end{itemize}
\end{Lemma}

\begin{proof}
Let us start with $\nu(\alpha,\mu)\in[0,\frac 12]$. Concerning the lower bound, employing the estimates \eqref{bessel_zero_bound} we can easily obtain  
\begin{align*}
	\sqrt{\lambda_{\alpha,\mu,k+1}}-\sqrt{\lambda_{\alpha,\mu,k}} = \frac{2-\alpha}{2}\left(j_{\nu(\alpha,\mu),k+1}-\jnk\right) \geq \frac{(2-\alpha)\pi}{2}\left(\frac{\nu(\alpha,\mu)}{4}+\frac 78\right) \geq \frac{7\pi}{16}(2-\alpha),
\end{align*}
since $\nu(\alpha,\mu)\geq 0$. Concerning the upper bound, thanks to Lemma \ref{eigen_gap_lemma} we immediately have that $j_{\nu(\alpha,\mu),k+1}-\jnk<\pi$, which clearly implies 
\begin{align*}
	\sqrt{\lambda_{\alpha,\mu,k+1}}-\sqrt{\lambda_{\alpha,\mu,k}} \leq \frac{\pi}{2}(2-\alpha).
\end{align*}

For $\nu(\alpha,\mu)\in[\frac 12,+\infty)$, instead, thanks again to Lemma \ref{eigen_gap_lemma} we have that $j_{\nu(\alpha,\mu),k+1}-\jnk>\pi$, which clearly implies 
\begin{align*}
	\sqrt{\lambda_{\alpha,\mu,k+1}}-\sqrt{\lambda_{\alpha,\mu,k}} \geq \frac{\pi}{2}(2-\alpha).
\end{align*}
Finally, the upper bound is again a consequence of Lemma \ref{eigen_gap_lemma}:
\begin{align*} 
	\sqrt{\lambda_{\alpha,\mu,k+1}}-\sqrt{\lambda_{\alpha,\mu,k}} = \frac{2-\alpha}{2}\left(j_{\nu(\alpha,\mu),k+1}-\jnk\right) \leq  \frac{2-\alpha}{2}\left(j_{\nu(\alpha,\mu),2}-j_{\nu(\alpha,\mu),1}\right),
\end{align*}
since the sequence $\left(j_{\nu(\alpha,\mu),k+1}-\jnk\right)_k$ is nonincreasing in that case. Observe that it is the best upper bound (valid for any $k \geq 1$) that one can obtain here.  
\end{proof}

Notice that, using the fact that $0 \leq \alpha <1$, one can deduce the following estimates that are also {\it uniform with respect to $\alpha$ and $\mu$} :
\begin{itemize} 
	\item when $\nu(\alpha,\mu)\in\left[0,\frac{1}{2}\right),$ 
	\begin{equation}\label{eigen-gap1_unif}
		\forall k \geq 1, \qquad \frac{7\pi}{16}\leq \sqrt{\lambda_{\alpha,\mu,k+1}}-\sqrt{\lambda_{\alpha,\mu,k}}
	 \leq \pi;
	\end{equation}
	\item when $\nu(\alpha,\mu)\in\left[\frac{1}{2},+\infty\right),$ 
	\begin{equation}\label{eigen-gap2_unif}
		\forall k \geq 1, \qquad \frac{\pi}{2}\leq \sqrt{\lambda_{\alpha,\mu,k+1}}-\sqrt{\lambda_{\alpha,\mu,k}}. 
	\end{equation}
\end{itemize}

On the other hand, let us observe that in the case $\nu(\alpha,\mu)\in\left[\frac{1}{2},+\infty\right)$, the upper estimate given in Lemma \ref{gap_lemma} is not satisfactory. Indeed, one can quote the following inequality from \cite{QuWong}:
\begin{equation*}
	\forall \nu>0, \forall n\geq 1, \qquad \nu - \frac{a_n}{2^{1/3}} \nu^{1/3} < j_{\nu,n} < \nu - \frac{a_n}{2^{1/3}} \nu^{1/3} + \frac{3}{20} a_n^2 \frac{2^{1/3}}{\nu^{1/3}},
\end{equation*}
where $a_n$ is the $n$-th negative zero of the Airy function. It follows that there exists $a >0$ such that
\begin{align*}
	j_{\nu,2} - j_{\nu,1} \sim a \, \nu ^{1/3} \quad \text{ as } \nu \to +\infty .
\end{align*}
Consequently, for any $\alpha \in [0,1)$, 
\begin{equation}\label{gammamax-explose}
	j_{\nu(\alpha, \mu),2} - j_{\nu(\alpha, \mu),1} \sim a \nu (\alpha, \mu) ^{1/3} \to + \infty  \quad \text{ as } \mu \to -\infty .
\end{equation}

Moreover, this upper estimate being the best possible one valid for any $k \geq 1$ (see the proof of Lemma \ref{gap_lemma}), it is of course not possible to improve it. 

Therefore, in order to get sharp estimates of the cost of controllability, it will be important to provide some \textit{better} upper estimates.  To this end, we will use the following complementary {\it asymptotic} estimates that is only valid for $k$ large enough but that has the advantage of being uniform with respect to $\alpha$ and $\mu$:  
\begin{Lemma}\label{asymp_gap_lemma}
When $\nu(\alpha,\mu)\in\left[\frac{1}{2},+\infty\right),$ for any $k > \nu (\alpha,\mu)$, we have
\begin{equation*}
	\sqrt{\lambda_{\alpha,\mu,k+1}}-\sqrt{\lambda_{\alpha,\mu,k}} \leq 2 \pi. 
\end{equation*}
\end{Lemma}

\begin{proof}
It directly follows from the definition of $\lambda_{\alpha,\mu,k}$ and Lemma 5.1 in \cite{CMV-the-cost-strong} that says that the zeros the Bessel functions satisfy 
\begin{align*}
	\forall \nu \geq \frac{1}{2},\quad \forall k >\nu, \qquad j_{\nu, k+1} -j_{\nu,k} \leq 2\pi.
\end{align*}
\end{proof}

\section{Well-posedness of the controllability problems}\label{well_pos_sec}

This section deals with the well-posedness of the models we are considering. To this end, let us first recall the functional framework associated to the purely degenerate operator $\mathbf P_{\alpha}$ (see for instance \cite{sicon2008}). 

\subsection{Functional framework}
For all $0 \leq \alpha <1$, we set 
\begin{align*}
	H^1_\alpha (0,1) := \Big\{ u \in L^2(0,1) \cap H^1_{loc} ((0,1]) \;\Big|\; x^{\alpha/2} u_x \in L^2 (0,1) \Big\}.
\end{align*}

Obviously, for any $u \in H^1_\alpha (0,1)$, the trace at $x=1$ exists. Moreover, in the case $0\leq \alpha <1$ (which is the one considered in this paper), it can be proved that the trace at $x=0$ also makes sense. This allows to introduce the space
\begin{align*}
	H^1_{\alpha,0} (0,1) := \Big\{ u \in H^1_\alpha (0,1)  \;\Big|\; u(0)=0=u(1) \Big\}.
\end{align*}

Next we introduce the functional setting associated to the  degenerate/singular operator $\mathbf P_{\alpha,\mu}$ (see \cite{vanco1}). For any $\mu \leq \mu(\alpha)$,  we define 
\begin{align*}
	H^{1, \mu}_\alpha (0,1) := \left\{ u \in L^2(0,1) \cap H^1_{loc} ((0,1]) \;\bigg|\; \int_0^1 \left(x^\alpha u_x^2 - \frac{\mu}{x^{2-\alpha}} u^2  
\right) dx <+\infty \right\}
\end{align*}
and 
\begin{align*}
	H^{1, \mu}_{\alpha,0} (0,1) := \Big\{ u \in H^{1,\mu}_\alpha (0,1) \;\Big|\; u(0)=0=u(1) \Big\}.
\end{align*}

In the case of a sub-critical parameter $\mu < \mu(\alpha)$, thanks to the generalized Hardy inequality \eqref{generalized-hardy}, it is easy to see that $H^{1,\mu}_{\alpha,0} (0,1)=H^1_{\alpha,0} (0,1)$. On the contrary, for the critical value $\mu = \mu(\alpha)$, the space is enlarged (see \cite{Va-Zu} for this observation in the case $\alpha=0$): 
\begin{align*}
	H^1_{\alpha,0} (0,1)  \underset{\not =}{\subset} H^{1,\mu(\alpha)}_{\alpha,0} (0,1).
\end{align*}
Next we define 
\begin{align*}
	H^{2, \mu}_{\alpha} (0,1) :=  \Big\{ u \in H^{1,\mu}_\alpha (0,1) \cap H^2_{loc} ((0,1]) \;\Big|\; (x^\alpha u_x)_x +\frac{\mu}{x^{2-\alpha}} u \in L^2(0,1) \Big\}.
\end{align*}
Finally,  the domain of the operator $\mathbf P_{\alpha,\mu}$ is given by 
\begin{align*}
	D(\mathbf P_{\alpha,\mu} ):= H^{2, \mu}_{\alpha} (0,1) \cap H^{1,\mu}_{\alpha,0} (0,1) .
\end{align*}

\subsection{Homogeneous boundary conditions and a source term}\label{sec-homo-bc}

Let us first consider the system with homogeneous boundary conditions and a source term
\begin{equation}\label{pb-auxi}
	\begin{cases}
		\displaystyle w_t -  (x^\alpha w_{x})_x -\frac{\mu}{x^{2-\alpha}}w =  f(x,t), & (x,t)\in Q
		\\
		w(0,t)=0, & t\in(0,T) 
		\\
		w(1,t)=0, & t\in(0,T)
		\\
		w(x,0)=w_0(x),  & x\in(0,1).
	\end{cases}
\end{equation}

Under the assumption \eqref{hyp-parameters} and for  any $w_0 \in L^2(0,1)$ and $f\in L^2((0,1)\times(0,T))$, problem \eqref{pb-auxi} is well-posed (see \cite{vanco1}) and we state the following definitions: 
\begin{Definition}
We have the following notions of solution:
\begin{itemize}
	\item[a)] Given $w_0 \in L^2(0,1)$ and $f \in L^2((0,1)\times (0,T))$, one defines the mild solution of \eqref{pb-auxi}
	\begin{align*}
		w \in \mathcal C^0 ([0,T]; L^2(0,1)) \cap L^2(0,T; H_{\alpha,0}^{1,\mu}(0,1))
	\end{align*}
	as the one given by the variation formula:
	\begin{align*}
		w(x,t) = e^{t \mathbf P_{\alpha,\mu}} w_0 + \int_0^t e^{(t-s) \mathbf P_{\alpha,\mu}} f(x,s) ds.
	\end{align*}
	\item[b)] We say that a function 
	\begin{align*}
		w \in \mathcal C^0 ([0,T]; H_{\alpha,0}^{1,\mu}(0,1)) \cap H^1(0,T; L^2(0,1)) \cap L^2(0,T; D(\mathbf P_{\alpha,\mu}))
	\end{align*}
	is a strict solution of \eqref{pb-auxi} if it satisfies the equation a.e. in $(0,1) \times (0,T)$ and the boundary and initial conditions for all $t\in [0,T]$ and $x \in [0,1]$. 
\end{itemize}
\end{Definition}

Notice that, if $w_0 \in H_{\alpha,0}^{1,\mu}(0,1)$, then the mild solution of \eqref{pb-auxi} is also the unique strict solution. 

\subsection{Non homogeneous boundary condition at $x=1$}

Next we turn to the boundary value problem  \eqref{pb-x-1}. To define the solution of \eqref{pb-x-1}, we transform it  into a problem with homogeneous boundary conditions and a source term. Let us introduce
\begin{align*}
	\forall x \in [ 0, 1 ], \qquad  p(x):= x^{q} \quad \text{ where }  \quad q:= \frac{1-\alpha}{2}+ \sqrt{\mu(\alpha)-\mu}.
\end{align*}
Observe that $p(0)=0$, $p(1)=1$ and 
\begin{align*}
	(x^\alpha p')'(x) + \frac{\mu}{x^{2-\alpha}} p(x) =0.
\end{align*}
Formally, if $u$ is a solution of \eqref{pb-x-1}, then the function defined by
\begin{equation}\label{def-v-prelim}
	v(x,t)=u(x,t)-\frac{p(x)}{p(1)}H(t) = u(x,t) - x^{q} H(t)
\end{equation}
is solution of 
\begin{equation}\label{pb-prelim}
	\begin{cases}
		\displaystyle v_t -  (x^\alpha v_{x})_x -\frac{\mu}{x^{2-\alpha}}v =  -\frac{p(x)}{p(1)} H'(t), & (x,t)\in Q
		\\
		v(0,t)=0, & t\in(0,T)
		\\
		v(1,t)=0, & t\in(0,T)
		\\
		\displaystyle v(x,0)=u_0(x) -\frac{p(x)}{p(1)} H(0), & x\in(0,1).
	\end{cases}
\end{equation}
Reciprocally, given $h \in L^2(0,T)$, consider the solution of 
\begin{align*}
	\begin{cases}
		\displaystyle v_t - (x^\alpha v_{x})_x  -\frac{\mu}{x^{2-\alpha}}v =  -\frac{p(x)}{p(1)} h(t), & (x,t)\in Q
		\\
		v(0,t)=0,  & t\in(0,T)
		\\
		v(1,t)=0, & t\in(0,T)
		\\
		v(x,0)=v_0(x), & x\in(0,1).
	\end{cases}
\end{align*}
Then the function $u$ defined by
\begin{equation*}
	u(x,t)=v(x,t)+\frac{p(x)}{p(1)} \int_0^t h (\tau) d\tau
\end{equation*}
satisfies 
\begin{equation*}
	\begin{cases}
		\displaystyle u_t -  (x^\alpha u_{x})_x -\frac{\mu}{x^{2-\alpha}}u =0, & (x,t)\in Q
		\\
		u(0,t)=0, & t\in(0,T)
		\\
		\displaystyle u(1,t)=\int_0^t h (\tau) d\tau, & t\in(0,T)
		\\
		u(x,0)=v_0(x), & x\in(0,1).
	\end{cases}
\end{equation*}

Let now $H$ be given in $H^1(0,T)$. The results of section \ref{sec-homo-bc} apply in particular to problem \eqref{pb-auxi} when one chooses 
\begin{equation}\label{def-f-v0}
	f(x,t)=  - \frac{p(x)}{p(1)} H'(t) \quad \text{ and } \quad v_0(x)=u_0(x) - \frac{p(x)}{p(1)}H(0).
\end{equation}
This allows us to define in a suitable way the solution of \eqref{pb-x-1}:

\begin{Definition}
We have the following notions of solution:
\begin{itemize}
	\item[a)] We say that  $ u \in \mathcal C^0 ([0,T]; L^2(0,1)) \cap L^2(0,T; H^{1,\mu}_{\alpha}(0,1))$ is the mild solution of \eqref{pb-x-1} if $v$ defined by \eqref{def-v-prelim} and \eqref{def-f-v0} is the mild solution of \eqref{pb-prelim}.

	\item[b)] We say that $ u \in \mathcal C^0 ([0,T]; H^{1,\mu}_{\alpha}(0,1)) \cap H^1(0,T; L^2(0,1)) \cap L^2(0,T; H^{2,\mu}_{\alpha}(0,1))$ is the strict solution of \eqref{pb-x-1} if $v$ defined by \eqref{def-v-prelim} and \eqref{def-f-v0} is the strict solution of \eqref{pb-prelim}. 
\end{itemize}
\end{Definition}

We  deduce
\begin{Proposition} \label{prop-wp}
Assume that $0\leq \alpha<1$ and $\mu \leq \mu(\alpha)$. 
\begin{itemize}
	\item[a)] Given $u_0 \in L^2(0,1)$ and $H \in H^1(0,T)$, problem \eqref{pb-x-1} admits a unique mild solution.

	\item[b)] Given $u_0 \in H^{1,\mu}_{\alpha}(0,1)$ such that $u_0(0)=0$ and $H \in H^1(0,T)$ such that $u_0(1)=H(0)$, problem \eqref{pb-x-1} admits a unique strict solution. In particular, this holds true when $u_0 \in H^{1,\mu}_{\alpha,0}(0,1)$ and  $H \in H^1(0,T)$ is such that $H(0)=0$.
\end{itemize}
\end{Proposition}

The proof of Proposition \ref{prop-wp} follows immediately noticing that 
\begin{align*}
	\widetilde H(x,t):= \frac{p(x)}{p(1)} H(t)
\end{align*}
satisfies 
\begin{align*}
	\widetilde H \in \mathcal C^0 ([0,T]; H^{1,\mu}_{\alpha}(0,1)) \cap H^1(0,T; L^2(0,1)) \cap L^2(0,T; H^{2,\mu}_{\alpha}(0,1)).
\end{align*}

\subsection{Non homogeneous boundary condition at $x=0$}

Finally we study to the boundary value problem  \eqref{pb-x-0}. To define its solution, as we did for \eqref{pb-x-1} before, we transform it into a problem with homogeneous boundary conditions and a source term. Let us introduce
\begin{align}\label{qdef}
	\forall x \in [0,1], \qquad  p(x):= 1-x^{q} \quad \text{ where } \quad q:= 2\sqrt{\mu(\alpha)-\mu}.
\end{align}

Observe that $q=0$ in the critical case $\mu=\mu(\alpha)$. (See also Remark \ref{rq-q} later). So we assume here that $\mu < \mu(\alpha)$. Then notice that $p(0)=1$, $p(1)=0$. Moreover, one can readily check that 
\begin{align}\label{Label}
	\big[x^\alpha(x^\gamma p)'\big]'(x) + \frac{\mu}{x^{2-\alpha-\gamma}} p(x) =0,
\end{align}
where $\gamma$ is the parameter introduced in \eqref{def-gamma}. Formally, if $u$ is a solution of \eqref{pb-x-0}, then the function defined by
\begin{equation}\label{def-v-prelim-bis}
	v(x,t)=u(x,t)-x^\gamma\frac{p(x)}{p(0)}H(t) = u(x,t) - x^\gamma(1-x^q) H(t)
\end{equation}
is solution of 
\begin{equation}\label{pb-prelim-bis}
	\begin{cases}
		\displaystyle v_t -  (x^\alpha v_{x})_x -\frac{\mu}{x^{2-\alpha}}v = F(x,t), & (x,t)\in Q
		\\
		v(0,t)=0, & t\in(0,T)
		\\
		v(1,t)=0, & t\in(0,T)
		\\
		\displaystyle v(x,0)=v_0(x), & x\in(0,1),
	\end{cases}
\end{equation}
where we denoted
\begin{align*}
	F(x,t):=-x^\gamma\frac{p(x)}{p(0)} H'(t) \quad \text{ and } \quad v_0(x):=u_0(x) -x^\gamma\frac{p(x)}{p(0)} H(0).
\end{align*}

Observe that \eqref{pb-x-0} actually implies $(x^{-\gamma}v)(0,t)=0$ which, in particular, gives $v(0,t)=0$ as written in \eqref{pb-prelim-bis}. Indeed, notice that $v$ in \eqref{def-v-prelim-bis} is given explicitly by 
\begin{align*}
	v(x,t) = \sum_{k\geq 1} v_k(t)\Phi_k(x),
\end{align*}
with 
\begin{align*}
	&v_k(t):=v_{k,0}e^{-\lambda_k t} + \int_0^t F_k(s)e^{-\lambda_k(t-s)}\,ds
	\\
	&v_{k,0} := \int_0^1 v_0(x)\Phi_k(x)\,dx
	\\
	&F_k(t) = \int_0^1 F(x,t)\Phi_k(x)\,dx.
\end{align*}
Then, as $x\to 0$ we have
\begin{align*}
	x^{-\gamma}v(x,t) &= \sum_{k\geq 1} v_k(t) x^{\frac{1-\alpha}{2}-\gamma}J_\nu\left(\jnk x^{\frac{2-\alpha}{2}}\right) = \sum_{k\geq 1} v_k(t) x^{\sqrt{\mu(\alpha)-\mu}}J_\nu\left(\jnk x^{\frac{2-\alpha}{2}}\right)\to 0.
\end{align*}
In view of that, we get 
\begin{align*}
	x^{-\gamma}u(x,t) = x^{-\gamma}v(x,t)+\frac{p(x)}{p(0)}H(t)\to H(t),\;\textrm{ as }\; x\to 0.
\end{align*}
Reciprocally, given $h \in L^2(0,T)$, consider the solution of 
\begin{align*}
	\begin{cases}
		\displaystyle v_t - (x^\alpha v_{x})_x  -\frac{\mu}{x^{2-\alpha}}v = -x^\gamma\frac{p(x)}{p(0)} h(t), & (x,t)\in Q
		\\
		v(0,t)=0, & t\in(0,T)
		\\
		v(1,t)=0, & t\in(0,T)
		\\
		v(x,0)=v_0(x), & x\in(0,1).
	\end{cases}
\end{align*}
Then the function $u$ defined by
\begin{equation*}
	u(x,t)=v(x,t)+x^\gamma\frac{p(x)}{p(0)} \int_0^t h (\tau) d\tau
\end{equation*}
satisfies 
\begin{equation*}
	\begin{cases}
		\displaystyle u_t -  (x^\alpha u_{x})_x -\frac{\mu}{x^{2-\alpha}}u =0, & (x,t)\in Q
		\\[5pt]
		\displaystyle(x^{-\gamma}u)(0,t)=\int_0^t h (\tau) d\tau, & t\in(0,T)
		\\[5pt]
		u(1,t)= 0, & t\in(0,T)
		\\[5pt]
		u(x,0)=v_0(x), & x\in(0,1).
	\end{cases}
\end{equation*}

Let now $H$ be given in $H^1(0,T)$. The results of section \ref{sec-homo-bc} apply in particular to problem \eqref{pb-auxi} when one chooses 
\begin{equation}\label{def-f-v0-bis}
	f(x,t)=  - x^\gamma\frac{p(x)}{p(0)} H'(t) \quad \text{ and } \quad v_0(x)=u_0(x) - x^\gamma\frac{p(x)}{p(0)}H(0).
\end{equation}
This allows us to define in a suitable way the solution of \eqref{pb-x-0}:

\begin{Definition}
We have the following notions of solution:
\begin{itemize}	
	\item[a)] We say that $ u \in \mathcal C^0 ([0,T]; L^2(0,1)) \cap L^2(0,T; H^{1,\mu}_{\alpha}(0,1))$ is the mild solution of \eqref{pb-x-0} if $v$ defined by \eqref{def-v-prelim-bis} and \eqref{def-f-v0-bis} is the mild solution of \eqref{pb-prelim-bis}.
	
	\item[b)] We say that $ u \in \mathcal C^0 ([0,T]; H^{1,\mu}_{\alpha}(0,1)) \cap H^1(0,T; L^2(0,1)) \cap L^2(0,T; H^{2,\mu}_{\alpha}(0,1))$ is the strict solution of \eqref{pb-x-0} if $v$ defined by \eqref{def-v-prelim-bis} and \eqref{def-f-v0-bis} is the strict solution of \eqref{pb-prelim-bis}. 
\end{itemize}	
\end{Definition}

We  deduce
\begin{Proposition} 
Assume that $0\leq \alpha<1$ and $\mu < \mu(\alpha)$. 
\begin{itemize}
	\item[a)] Given $u_0 \in L^2(0,1)$ and $H \in H^1(0,T)$, problem \eqref{pb-x-0} admits a unique mild solution.
	
	\item[b)] Given $u_0 \in H^{1,\mu}_{\alpha}(0,1)$  and $H \in H^1(0,T)$ such that $(x^{-\gamma}u_0)(0)=H(0)$ and $u_0(1)=0$, problem \eqref{pb-x-0} admits a unique strict solution. In particular, this holds true when $u_0 \in H^{1,\mu}_{\alpha}(0,1)$ such that $(x^{-\gamma}u_0)(0)=u_0(1)=0$.
\end{itemize}	
\end{Proposition}

The proof of Proposition \ref{prop-wp} follows immediately noticing that 
\begin{align*}
	\widetilde H(x,t):= x^\gamma\frac{p(x)}{p(0)} H(t)
\end{align*}
satisfies 
\begin{align*}
	\widetilde H \in \mathcal C^0 ([0,T]; H^{1,\mu}_{\alpha}(0,1)) \cap H^1(0,T; L^2(0,1)) \cap L^2(0,T; H^{2,\mu}_{\alpha}(0,1)).
\end{align*}
\begin{Remark}  \label{rq-q}
As a final remark we observe that, when $\mu=\mu(\alpha)$, the value of $q$ that we defined in \eqref{qdef} is zero. This means that, in the case of critical potentials, the change of variables introduced for defining the solution to our problem is the trivial one. Therefore, in what follows, when dealing with \eqref{pb-x-0} we shall always assume $\mu< \mu(\alpha)$. Notice, however, that this assumption is not a limitation. Indeed, for critical potentials we
do not expect our equation \eqref{pb-x-0} to be well posed, at least not with the boundary conditions that we are imposing. This behavior had already been observed in \cite{biccari2} for purely singular operators ($\alpha=0$), and a more detailed discussion on this point can be found in the Appendix A of that mentioned work.
\end{Remark}

\section{Proof of Theorem \ref{control_thm1}}\label{sec-pf-control_thm1}

This Section is devoted to the proof of our first result Theorem \ref{control_thm1} on the boundary controllability for \eqref{pb-x-1}. 

The proof will employ the classical moment method (see \cite{FR1,FR2}). This procedure is based on the explicit construction of the control $H$, given in terms of a family $(\sigma_{\alpha,\mu,m}(t))_{m\geq 1}$ {\it biorthogonal} in $L^2(0,T)$ to the family of real exponential $( e^{\lambda_{\alpha,\mu,n}})_{n\geq1}$, that is
\begin{equation}\label{biortho+1}
	\forall m, n \geq 0, \quad \int _0 ^T \sigma_{\alpha,\mu,m} (t) e^{\lambda_{\alpha,\mu,n}t} \, dt = \delta _{mn}=
	\begin{cases} 1 \text{ if } m=n , \\ 0 \text{ if } m \neq n . \end{cases}
\end{equation}

In order to show the existence of such a sequence, we will use \cite[Theorem 2.4]{cost-weak}, whose proof has been inspired by the works of Seidman-Avdonin-Ivanov \cite{Seid-Avdon} and Tucsnak-Tenenbaum \cite{Tucsnak}. In this part, it will be fundamental that the eigenvalues associated to our problem fulfill the gap conditions
\begin{align*}
	\forall n\geq 1, \qquad \sqrt{\lambda_{\alpha,\mu,n+1}} - \sqrt{\lambda_{\alpha,\mu,n}} \geq \gamma_{\text{min}}.
\end{align*}

Furthermore, to define properly the control $H$, we also need to provide some sharp lower bound of the norm $\Vert \sigma_{\alpha,\mu,m}(t)_{m\geq 1} \Vert_{L^2(0,T)}$, which will be obtained as a consequence of this second spectral estimate
\begin{equation}\label{hyp-gmax-intro}
	\forall n\geq 1, \qquad \sqrt{\lambda_{\alpha,\mu,n+1}} - \sqrt{\lambda_{\alpha,\mu,n}} \leq \gamma_{\text{max}}.
\end{equation}

When $\nu(\alpha,\mu)\in\left[0,\frac{1}{2}\right]$, \eqref{hyp-gmax-intro} holds true for some $\gamma_{max}$ that is independent of $\alpha$ and $\mu$. Here we will use \cite[Theorem 2.5]{cost-weak}, inspired from Guichal \cite{Guichal}.

When $\nu(\alpha,\mu)\in\left[\frac{1}{2},+\infty\right)$, \eqref{hyp-gmax-intro} still holds true but with $\gamma_{max}$ that tends to $+\infty$ as $\mu \to -\infty$. So one could still use \cite[Theorem 2.5]{cost-weak} but this would not give a sharp estimate. For this reason, we complement \eqref{hyp-gmax-intro} by the \textit{better} asymptotic estimate given in Lemma \ref{asymp_gap_lemma}. Then we will use \cite[Theorem 2.2]{CMV-biortho-general}. 

Therefore, according to the above discussion, our proof will be organized into the following steps: 
\begin{itemize}
	\item[•] \textbf{Step 1.} Following the classical approach of \cite{FR1,FR2}, we first reduce our control problem to a moment problem.
	\item[•] \textbf{Step 2.} We  give a formal solution, using the properties of the spectrum of the operator $\mathbf{P}_{\alpha,\mu}$.
	\item[•] \textbf{Step 3.} We prove the existence of the control, its regularity (in $H^1(0,T)$) and also give an upper bound of the cost of controllability.
	\item[•] \textbf{Step 4.} We finally derive a lower bound of the cost of controllability. 
\end{itemize}

Let $\alpha$ and $\mu$ be given such that $0\leq \alpha<1$ and $\mu \leq \mu(\alpha)$. For simplicity in the notations, we denote in the following by $\Phi_k$ (instead of $\Phi_{\alpha,\mu,k}$) and $\lambda_k$ (instead of $\lambda_{\alpha,\mu,k}$) the eigen-elements given by Proposition \ref{prop-eigenelements}, and by $\sigma_m$ (instead of $\sigma_{\alpha,\mu,m}$) the biorthogonal family. Besides, also for simplicity in the notations, we will denote in a generic way by $\C$ all the constants (independent of $k$, $\alpha$, $\mu$ and $T$) that appear in the calculus. We stress that the value of $\C$ may change from line to line.

\subsection{Reduction to a moment problem} 
In this part, we treat the problem with formal computations. We will present a rigorous justification in a second moment. 

Let us start expanding the initial condition $u_0\in L^2(0,1)$ in the basis of the eigenfunctions $(\Phi_k)_{k\geq 1}$. Indeed, we know that there exists a sequence $(\,\rho_k^0)_{k\geq 1}\in \ell^2(\NN^*)$ such that, for all $x\in(0,1)$,
\begin{align*}
	u_0(x)=\sum_{k\geq 1}\rho_k^0\Phi_k(x), \quad \rho_k^0:=\int_0^1 u_0(x)\Phi_k(x)\,dx,\quad k\geq 1.
\end{align*}
Next, we expand also the solution $u$ to \eqref{pb-x-1} as
\begin{align*}
	u(x,t)=\sum_{k\geq 1}\beta_k(t)\Phi_k(x), \quad (x,t)\in Q,
\end{align*}
with
\begin{align*}
	\beta_k(t):=\int_0^1 u(x,t)\Phi_k(x)\,dx, \quad k\geq 1
\end{align*}
and
\begin{align*}
	\sum_{k\geq 1}\beta_k(t)^2<+\infty.
\end{align*}
Therefore, the controllability condition $u(x,T)=0$ becomes 
\begin{align}\label{controllability_moment}
\forall k\geq 1,\;\;\; \beta_k(T)=0.
\end{align}
Moreover, we notice that the function $v_k(x,t):=\Phi_k(x)e^{\lambda_k(t-T)}$ solves the adjoint problem
\begin{align}\label{heat_hardy_adj}
\renewcommand*{\arraystretch}{1.3}
	\left\{\begin{array}{ll}
		\displaystyle v_{k,t}+(x^\alpha v_{k,x})_x+\frac{\mu}{x^{2-\alpha}}v_k=0, & (x,t)\in Q
		\\ 
		v_k(0,t)=v_k(1,t)=0, & t\in (0,T).
	\end{array}\right.
\end{align}
Combining \eqref{pb-x-1} and \eqref{heat_hardy_adj} we obtain
\begin{align*}
	0=& \int_Q \bigg[v_k\left(u_t-(x^\alpha u_x)_x-\frac{\mu}{x^{2-\alpha}}u\right)+u\left(v_{k,t}+(x^\alpha v_{k,x})_x+\frac{\mu}{x^{2-\alpha}}v_k\right)\bigg]\,dxdt
	\\
	=& \int_0^1 v_ku\,\Big|_0^T\,dx - \int_0^T x^\alpha u_xv_k\,\Big|_0^1\,dt + \int_0^T x^\alpha v_{k,x}u\,\Big|_0^1\,dt
	\\
	=& \int_0^1 v_k(x,T)u(x,T)\,dx - \int_0^1 v_k(x,0)u_0(x)\,dx + \int_0^T H(t)v_{k,x}(1,t)\,dt 
	\\
	=& \int_0^1 u(x,T)\Phi_k(x)\,dx - e^{-\lambda_k T}\int_0^1 u_0(x)\Phi_k(x)\,dx + e^{-\lambda_kT}\Phi_k'(1)\int_0^T H(t)e^{\lambda_k t}\,dt
	\\
	=& \,\beta_k(T) - \rho_k^0e^{-\lambda_k T} + e^{-\lambda_k T}\Phi_k'(1)\int_0^T H(t)e^{\lambda_kt}\,dt.
\end{align*}
Then,  \eqref{controllability_moment} yields
\begin{align}\label{moment_cond}
\forall k\geq 1,\;\;\;\Phi_k'(1)\int_0^T H(t)e^{\lambda_kt}\,dt = \rho_k^0. 
\end{align}

On the other hand, since we are looking for a solution of the moment problem belonging to $H^1(0,T)$, instead of \eqref{moment_cond} we would rather be interested in a condition involving the derivative of the function $H$. This condition can be obtained integrating by parts in \eqref{moment_cond}, as follows
\begin{align*}
	\int_0^T H(t)e^{\lambda_kt}\,dt = \frac{1}{\lambda_k}H(t)e^{\lambda_k t}\,\bigg|_0^T-\frac{1}{\lambda_k}\int_0^T H'(t)e^{\lambda_k t}\,dt.
\end{align*}
Therefore, $H'(t)$ has to satisfy
\begin{align}\label{moment_cond_H1}
	\forall k\geq 1,\;\;\;-\frac{\Phi_k'(1)}{\lambda_k}\int_0^T H'(t)e^{\lambda_kt}\,dt = \rho_k^0 -\frac{\Phi_k'(1)}{\lambda_k}\left(H(T)e^{\lambda_k T}-H(0)\right).
\end{align}
We will provide a solution to the above problem which satisfies $H(0)=H(T)=0$. 

\subsection{Formal solution of the moment problem}
We exhibit  here a formal solution of the moment problem \eqref{moment_cond_H1}. 

\subsubsection{Formal definition of the control $H$}
Set artificially $ \lambda_{ 0} := 0 ,$ so that we have now a sequence $(\lambda_{k})_{k \geq 0}$. We assume for the moment that we are able to construct a family $(\sigma_{m})_{m \geq 0}$ of functions $\sigma_{m} \in L^2(0,T)$, which is biorthogonal to the family $(e^{\lambda_{n} t })_{n\geq 0}$. Observe that for $n=0$, using $\lambda_0 = 0 $, \eqref{biortho+1} implies
\begin{equation}\label{sigma_int}
	 \forall m \geq 1, \quad \int _0 ^T \sigma_{ m} (t) \, dt = 0.
\end{equation}
Then let us define the function $H$ as follows:
\begin{align}\label{H_def}
	H(t):=\int_0^t K(s)\,ds, \;\;\;\textrm{ with }\;K(t):=-\sum_{k\geq 1} \frac{\lambda_k}{\Phi_k'(1)}\rho_k^0\sigma_k(t).
\end{align}

It is straightforward that, if $K\in L^2(0,T)$, then $H\in H^1(0,T)$ with $H(0)=0$ and $H'(t)=K(t)$. Moreover thanks to \eqref{sigma_int} we have, at least formally,
\begin{align*}
	H(T)= -\int_0^T \sum_{k\geq 1} \frac{\lambda_k}{\Phi_k'(1)}\rho_k^0\sigma_k(s)\,ds = - \sum_{k\geq 1} \frac{\lambda_k}{\Phi_k'(1)}\rho_k^0\int_0^T\sigma_k(s)\,ds=0.
\end{align*}
Finally,
\begin{align*}
	-\frac{\Phi_k'(1)}{\lambda_k}\int_0^T H'(t)e^{\lambda_k t}\,dt &= -\frac{\Phi_k'(1)}{\lambda_k}\int_0^T K(t)e^{\lambda_k t}\,dt = \frac{\Phi_k'(1)}{\lambda_k}\int_0^T \left(\sum_{\ell\geq 1} \frac{\lambda_{\ell}}{\Phi_{\ell}'(1)}\rho_{\ell}^0\sigma_{\ell}(t)\right)e^{\lambda_k t}\,dt
	\\
	&= \frac{\Phi_k'(1)}{\lambda_k}\sum_{\ell\geq 1}\frac{\lambda_{\ell}}{\Phi_{\ell}'(1)}\rho_{\ell}^0\int_0^T \sigma_{\ell}(t)e^{\lambda_k t}\,dt = \frac{\Phi_k'(1)}{\lambda_k}\sum_{\ell\geq 1}\frac{\lambda_{\ell}}{\Phi_{\ell}'(1)}\rho_{\ell}^0\delta_{k,\ell} = \rho_k^0,
\end{align*}
and the moment problem \eqref{moment_cond_H1} is formally satisfied.

\subsubsection{If regular, the control $H$ drives the solution from $u_0$ to zero}
Let us assume for now that $K \in L^2(0,T)$ (and, consequently $H$ introduced in \eqref{H_def} belongs to $H^1(0,T)$). We show here that $H$ is able to drive the solution to \eqref{pb-x-1} from the initial state $u_0$ to zero in time $T$. To this end, let us remind the change of variables 
\begin{align*}
	v(x,t):=u(x,t)-\frac{p(x)}{p(1)}H(t),\;\;\; p(x):=x^q, \;\;\; q=\frac{1-\alpha}{2}+\sqrt{\mu(\alpha)-\mu},
\end{align*}
that transforms our original equation \eqref{pb-x-1} in
\begin{align*}
\renewcommand*{\arraystretch}{1.3}
	\left\{\begin{array}{ll}
		\displaystyle v_t-(x^\alpha v_x)_x-\frac{\mu}{x^{2-\alpha}}v = -\frac{p(x)}{p(1)}K(t), & (x,t)\in Q
		\\ 
		v(0,t)=v(1,t)=0, & t\in (0,T)
		\\ 
		v(x,0)=u_0(x), & x\in (0,1)
	\end{array}\right.
\end{align*} 
Now, for a fixed $\varepsilon>0$ we have
\begin{align*}
	\int_{\varepsilon}^T&\int_0^1 -\frac{p(x)}{p(1)}K(t)\Phi_k(x)e^{\lambda_k t}\,dxdt 
	\\
	&= \int_{\varepsilon}^T\int_0^1 \left(v_t-(x^\alpha v_x)_x-\frac{\mu}{x^{2-\alpha}}v\right)\Phi_k(x)e^{\lambda_k t}\,dxdt
	\\
	&= \int_0^1 v\Phi_ke^{\lambda_k t}\,\Big|_{\varepsilon}^T\,dx + \int_{\varepsilon}^T\int_0^1 v\left(-(x^\alpha\Phi_k')'-\frac{\mu}{x^{2-\alpha}}\Phi_k-\lambda_k\Phi_k\right)e^{\lambda_k t}\,dxdt
	\\
	&=e^{\lambda_k T}\int_0^1 v(x,T)\Phi_k(x)\,dx - e^{\lambda_k \varepsilon}\int_0^1 v(x,\varepsilon)\Phi_k(x)\,dx.
\end{align*}
Hence, taking the limit for $\varepsilon\to 0^+$ we find
\begin{align*}
	\int_Q -\frac{p(x)}{p(1)}K(t)\Phi_k(x)e^{\lambda_k t}\,dxdt = e^{\lambda_k T}\int_0^1 v(x,T)\Phi_k(x)\,dx - \rho_k^0.
\end{align*}
From this last identity and \eqref{H_def}, it immediately follows
\begin{align*}
	e^{\lambda_k T}\int_0^1 v(x,T)\Phi_k(x)\,dx &= \rho_k^0 + \left(\int_0^T K(t)e^{\lambda_k t}\,dt\right)\left(\int_0^1 -\frac{p(x)}{p(1)}\Phi_k(x)\,dx\right)
	\\
	&= \rho_k^0 -\frac{\lambda_k}{\Phi_k'(1)}\rho_k^0\int_0^1 -\frac{p(x)}{p(1)}\Phi_k(x)\,dx.
\end{align*}
Moreover, 
\begin{align*}
	\int_0^1 & -\frac{p(x)}{p(1)}\Phi_k(x)\,dx 
	\\
	&= \frac{1}{\lambda_k}\int_0^1 -\frac{p(x)}{p(1)}\lambda_k\Phi_k(x)\,dx = \frac{1}{\lambda_k}\int_0^1\frac{p(x)}{p(1)}\left((x^\alpha\Phi_k'(x))'+\frac{\mu}{x^{2-\alpha}}\Phi_k(x)\right)\,dx
	\\
	&= \frac{1}{\lambda_k}\frac{p(x)}{p(1)}x^{\alpha}\Phi_k'(x)\,\bigg|_0^1 - \frac{1}{\lambda_k}\int_0^1 \frac{p'(x)}{p(1)}x^{\alpha}\Phi_k'(x)\,dx + \frac{1}{\lambda_k}\int_0^1 \frac{p(x)}{p(1)}\frac{\mu}{x^{2-\alpha}}\Phi_k(x)\,dx 
	\\
	&=\frac{\Phi_k'(1)}{\lambda_k} - \frac{1}{\lambda_k}\frac{p'(x)}{p(1)}x^{\alpha}\Phi_k(x)\,\bigg|_0^1 + \frac{1}{\lambda_k}\int_0^1 \left[\left(x^{\alpha}\frac{p'(x)}{p(1)}\right)'+ \mu x^{\alpha-2}\frac{p(x)}{p(1)}\right]\Phi_k(x)\,dx 
	\\
	&=\frac{\Phi_k'(1)}{\lambda_k} + \frac{1}{\lambda_k p(1)}\int_0^1 \Big[\left(x^{\alpha}p'(x)\right)'+ \mu x^{\alpha-2}p(x)\Big]\Phi_k(x)\,dx = \frac{\Phi_k'(1)}{\lambda_k},
\end{align*}
since from the definition of $p(x)$ it is straightforward to check that
\begin{align*}
	\left(x^{\alpha}p'(x)\right)'+ \mu x^{\alpha-2}p(x)=0. 
\end{align*} 
Hence, we get
\begin{align*}
	e^{\lambda_k T}\int_0^1 v(x,T)\Phi_k(x)\,dx = 0,
\end{align*}
which of course implies $v(x,T) = 0$ and, since $H(T)=0$, we can finally conclude that
\begin{align*}
	u(x,T)=v(x,T)+\frac{p(x)}{p(1)}H(T)=0.
\end{align*}

At this stage, in order to prove point (i) of Theorem \ref{control_thm1}, it remains to prove the existence of a suitable biorthogonal family and to show that $K$ belongs to $L^2(0,T)$. This will be done in the next subsection together with the obtention of the upper bound of the cost of controllability.

\subsection{Existence of the control, $H^1$ regularity  and upper bound of the cost of controllability}

\subsubsection{Existence of a suitable biorthogonal family}
We will use the following result.
\begin{Theorem}(see \cite[Theorem 2.4]{cost-weak})\label{thm-biortho1-gen}
Assume that for all $k\geq 0$, $\lambda_k \geq 0, $ and that there is some $\gamma _{\text{min}}>0$ such that
\begin{equation*}
	\forall k \geq 0, \quad \sqrt{\lambda _{k+1}} - \sqrt{\lambda _{k}}  \geq \gamma _{\text{min}} .
\end{equation*}
Then there exists a family $(\sigma _{m} )_{m\geq 0}$ which is biorthogonal to the family $(e^{\lambda _{k}t})_{k\geq 0}$ in $L^2(0,T)$. Moreover, there exists some universal constant $\C_u$ independent of $T$, $\gamma _{\text{min}}$ and $m$ such that, for all $m\geq 0$, we have
\begin{equation}\label{*famillebi_qm-norme-gen-bis}
	\Vert \sigma _{m}  \Vert _{L^2(0,T)} ^2  \leq \C_u e^{-2\lambda _{m} T} e^{\C_u \frac{\sqrt{\lambda _{m}}}{ \gamma _{min}} } e^{\frac{\C_u}{\gamma _{min} ^2T}} B^\star(T,\gamma _{min}),
\end{equation}
with
\begin{equation}\label{eq(Betoile}
	B^\star(T,\gamma _{min}) =\frac{\C_u}{T}\text{max} \, \left\{ T \gamma _{min} ^2, \frac{1}{T \gamma _{min} ^2}\right\} .
\end{equation}
\end{Theorem}

\begin{Remark} {\rm \cite[Theorem 2.4]{cost-weak} is formulated in the following way:
\begin{equation*}
	\Vert \sigma _{m}  \Vert _{L^2(0,T)} ^2  \leq \C_u e^{-2\lambda _{m} T} e^{\C_u \frac{\sqrt{\lambda _{m}}}{ \gamma _{min}} }  B(T,\gamma _{min}),
\end{equation*}
with
\begin{equation*}
	B(T,\gamma _{min}) = \begin{cases} \Bigl( \frac{1}{T} + \frac{1}{T^2 \gamma _{min} ^2} \Bigr) \, e^{\frac{\C_u}{\gamma _{min} ^2T}}  \quad & \text{ if } T \leq \gamma _{min}^{-2}, \\ \C_u \gamma _{min} ^2 \quad & \text{ if } T \geq \gamma_{min}^{-2}, \end{cases}
\end{equation*}
and this is clearly equivalent to \eqref{*famillebi_qm-norme-gen-bis}-\eqref{eq(Betoile}.
}
\end{Remark}

Using \eqref{eigen-gap1_unif} and \eqref{eigen-gap2_unif}, the eigenvalues of the problem satisfy for all $\mu \leq \mu(\alpha)$
\begin{align*}
	\forall k\geq 1, \quad \sqrt{\lambda _{k+1}} - \sqrt{\lambda _{k}} \geq \text{min} \, \left\{ \frac{7\pi}{16}, \frac{\pi}{2} \right\} =
\frac{7\pi}{16}.
\end{align*}
As before, define artificially $\lambda _{0} :=0.$ Then, for all $\mu \leq  \mu(\alpha)$, 
\begin{align*}
	\sqrt{\lambda _{1}} - \sqrt{\lambda _{0}} = \frac{2-\alpha}{2} j_{\nu(\alpha,\mu),1} \geq \frac{2-\alpha}{2}\cdot \frac{3\pi}{4} =\frac{3\pi}{8} (2-\alpha) \geq \frac{3\pi}{8},
\end{align*}
using the fact that, thanks to \eqref{bessel_zero_bound}, one can easily prove that $j_{\nu,1} \geq 3\pi/4$ for all $\nu \geq 0$ and next using the fact that $2-\alpha \geq 1$. Therefore we can apply Theorem \ref{thm-biortho1-gen} to the family $(e^{\lambda_{k}t})_{k\geq 0}$ provided that we choose 
\begin{align*}
	\gamma _{min} = \text{min} \left\{ \frac{7\pi}{16}, \frac{3\pi}{8} \right\} = \frac{3\pi}{8} .
\end{align*}
We obtain that there exists a family $(\sigma _{m} )_{m\geq 0}$ biorthogonal to $(e^{\lambda _{k}t})_{k\geq 0}$ in $L^2(0,T)$, and such that
\begin{equation}\label{majo-de-sigma}
	\Vert \sigma _{m}  \Vert _{L^2(0,T)} ^2  \leq \C e^{-2\lambda _{m} T} e^{\C \frac{\sqrt{\lambda _{m}}}{\gamma_{min}}}
\widetilde B(T) \leq \C e^{-2\lambda _{m} T} e^{\C \sqrt{\lambda _{m}} }\widetilde B(T) ,
\end{equation}
with  
\begin{align*}
	\widetilde B(T)= \text{max} \left\{1, \frac{1}{T^2} \right\} e^{\frac{\C}{T}}\; \text{ for all } T>0.
\end{align*}
The form of $\widetilde B(T)$ easily follows from the definition of $B(T, \gamma_{min})^\star$. 

\subsubsection{The control $f$ belongs to $H^1(0,T)$}
We have to check that the control $H$ defined as in \eqref{H_def} belongs to $H^1(0,T)$. To this end, we are going to prove, instead, that the function $K$ belongs to $L^2(0,T)$. From \eqref{H_def} we have 
\begin{align*}
	\norm{K}{L^2(0,T)} &= \norm{\sum_{k\geq 1}\frac{\lambda_k}{\Phi_k'(1)}\rho_k^0\sigma_k}{L^2(0,T)} \leq \sum_{k\geq 1} |\,\rho_k^0|\left|\frac{\lambda_k}{\Phi_k'(1)}\right|\norm{\sigma_k}{L^2(0,T)}.
\end{align*}
Let us compute the value of $\vert \Phi_k'(1) \vert$: we recall that 
\begin{align*}
	\Phi_k(x)=\C_k\, x^{\frac{1-\alpha}{2}}J_{\nu(\alpha,\mu)}\left(\jnk x^{\frac{2-\alpha}{2}}\right), \;\;\;\textrm{ with }\;\;\; 
\C_k=\frac{\sqrt{2-\alpha}}{|J_{\nu(\alpha,\mu)}'(\jnk)|}.
\end{align*}
Thus, a direct computation gives 
\begin{align*}
	\Phi_k'(x) = \frac{1-\alpha}{2}\C_k x^{-\frac{1+\alpha}{2}}J_{\nu(\alpha,\mu)}\left(\jnk x^{\frac{2-\alpha}{2}}\right) 
	+ \frac{2-\alpha}{2}\C_k\jnk x^{\frac{1-2\alpha}{2}}J_{\nu(\alpha,\mu)}'\left(\jnk x^{\frac{2-\alpha}{2}}\right).
\end{align*}
Therefore	
\begin{align}\label{eigen_deriv}
	|\Phi_k'(1)| &= \left|\frac{1-\alpha}{2}\C_k J_{\nu(\alpha,\mu)}(\jnk) + \frac{2-\alpha}{2}\C_k\jnk J_{\nu(\alpha,\mu)}'(\jnk)\right|
=\frac{(2-\alpha)^{\frac 32}}{2} \jnk.\notag
\\
\end{align}
Consequently, employing \eqref{eigen_deriv} and the explicit expression of the eigenvalues $\lambda_k$ we obtain
\begin{align*}
	\left|\frac{\lambda_k}{\Phi_k'(1)}\right| = \frac{\left(\frac{2-\alpha}{2}\right)^2\jnk^2}{\frac{(2-\alpha)^{\frac 32}}{2}\jnk} = \frac{\sqrt{2-\alpha}}{2}\jnk\leq \sqrt{2}\jnk. 
\end{align*}
Therefore,  we get
\begin{align*}
	\norm{K}{L^2(0,T)} \leq \sqrt{2} \sum_{k\geq 1} |\,\rho_k^0| \jnk \norm{\sigma_k}{L^2(0,T)} \leq \sqrt{2}\left( \sum_{k\geq 1} |\,\rho_k^0|^2 \right)^{1/2} \left( \sum_{k\geq 1} \jnk ^2 \norm{\sigma_k}{L^2(0,T)}^2 \right)^{1/2}.
\end{align*}

Using the explicit expression of $\lambda_k$, we get $\jnk ^2 = 4 \lambda_k/ (2-\alpha)^2 \leq 4 \lambda_k$ since $\alpha <1$.
Hence, using also the estimate \eqref{majo-de-sigma}, we deduce that 
\begin{align*}
	\norm{K}{L^2(0,T)} \leq \C\norm{u_0}{L^2(0,1)} \left( \sum_{k\geq 1} \lambda_k  e^{-2\lambda _k T} e^{\C \sqrt{\lambda _{k}} }
\widetilde B(T) \right)^{1/2},
\end{align*}
which is finite. This implies that $K \in L^2(0,T)$. Therefore we have $H\in H^1(0,T)$ with of course $H(0)=0$. And the fact that $H(T)=0$ follows from \eqref{sigma_int} with $m=0$. 

\subsubsection{Upper bound of the cost of controllability }\label{upp_bound_one_sec}
As shown before, the function $H$ defined in \eqref{H_def} is an admissible control. It follows that 
\begin{align*}
	\C _{ctr-bd} \leq \frac{\Vert H \Vert _{H^1 (0,T)}}{\Vert u_0 \Vert _{L^2 (0,1)} } \leq \C \frac{\Vert K \Vert _{L^2 (0,T)}}{\Vert u_0 \Vert _{L^2 (0,1)} }.
\end{align*}
Hence
\begin{equation*}
	\C _{ctr-bd} \leq \C \sqrt{\widetilde B(T) } \left( \sum _{k=1} ^\infty  \lambda _{k}  e^{-2\lambda _{k} T} e^{\C \sqrt{\lambda _{k}} } \right) ^{1/2}.
\end{equation*}
Then let us write 
\begin{align*}
	\C \sqrt{\lambda _{k}} \leq \lambda _{k}T + \frac{\C'}{T}.
\end{align*}
One deduces that 
\begin{align*} 
	\C _{ctr-bd} &\leq \C \sqrt{\widetilde B(T)} \left( \sum _{k=1} ^\infty \lambda_k e^{-\lambda_k T} e^{\frac{\C'}{T}}\right)^{1/2} 
	\\
	&\leq \C \sqrt{\widetilde B(T) } e^{\frac{\C}{T}} \left(\sum _{k=1} ^\infty \frac{(2-\alpha)^2}{4} (j_{\nu(\alpha, \mu),k})^2   e^{-\frac{(2-\alpha)^2}{4} j _{\nu(\alpha, \mu),k}^2 T} \right)^{1/2}
	\\
	&\leq C e^{\frac{\C}{T}}\left( \sum _{k=1} ^\infty (j_{\nu(\alpha, \mu),k}  )^2 e^{- \frac{(2-\alpha)^2}{4} j _{\nu(\alpha, \mu),k}^2 T}\right)^{1/2}.
\end{align*}
Next we use the following Lemma proved in \cite{CMV-the-cost-strong}:
\begin{Lemma} 
\label{lem-maj-serie}
There is some constant $\C>0$, independent of $\nu$ and of $Y$, such that :
\begin{equation*}
	\forall \nu \geq 0, \quad \forall Y>0, \qquad \sum _{k=1} ^\infty j_{\nu,k}^2  e^{-j _{\nu,k}^2 Y} \leq \C \frac{1+ \nu ^2 }{Y^{3/2}} e^{-(1+\nu ^2) \frac{Y}{\C}}.
\end{equation*}
\end{Lemma}

Applying Lemma \ref{lem-maj-serie} with $Y=\frac{(2-\alpha)^2}{4} T$ , it follows that
\begin{align*}
	\C_{ctr-bd} &\leq \C  e^{\frac{\C}{T}} \frac{1+ \nu (\alpha,\mu)}{T^{3/4}} e^{-\frac{1+\nu (\alpha,\mu)^2}{2\C} \frac{(2-\alpha)^2}{4} T} \leq \C e^{\frac{\C}{T}} [1+ \nu (\alpha,\mu)] e^{- \C (1+\nu (\alpha,\mu)^2)  T} 
	\\
	&\leq \C e^{\frac{\C}{T}} \left[1+ \frac{2}{2-\alpha} \sqrt{\mu(\alpha)-\mu} \right] e^{- \C \left( 1+\frac{4}{(2-\alpha)^2} (\mu(\alpha)-\mu) \right)  T} 
	\\
	&\leq \C e^{\frac{\C}{T}} \left[1+  \sqrt{\mu(\alpha)-\mu}\right] e^{- \C \left( 1+\sqrt{\mu(\alpha)-\mu} \right)^2 T}  .
\end{align*}

\subsection{Lower bound of the cost of controllability}
Let us fix $m\geq 1$ and let us choose $u_0 = \Phi _{m}$. Consider $H_m$ any control that drives the solution of \eqref{pb-x-1} to zero in time $T$. Then \eqref{moment_cond} reads as
\begin{align*}
	\forall k \geq 1, \quad \Phi_k'(1) \int _0 ^T  H_m(t)  e^{\lambda _{k} t} \, dt = \rho^0_k=\delta _{mk}.
\end{align*}
It follows that
\begin{align*}
	\forall k \geq 1, \quad \int _0 ^T  \Big( \Phi_m'(1)  H_m(t) \Big)  e^{\lambda _{k} t} \, dt =\delta _{mk}.
\end{align*}

In other words, the sequence $(\Phi_m'(1)  H_m)_{m\geq 1}$ is biorthogonal to the set $(e^{\lambda _{k} t})_{k\geq 1}$.  At this stage,  we will distinguish the two following cases:
\begin{align*}
	\nu(\alpha, \mu) \in \left[0, \frac{1}{2}\right] \quad \text{ that is, }\mu \in \left[ \frac{\alpha}{16} (3\alpha-4), \mu(\alpha)\right]
\end{align*}
and 
\begin{align*}
	\nu(\alpha, \mu) \in \left[\frac{1}{2}, +\infty\right) \quad \text{ that is, }\mu  \in \left(-\infty,  \frac{\alpha}{16} (3\alpha-4)\right].
\end{align*}

\subsubsection{Lower bound in the case  $\nu(\alpha, \mu) \in \left[0, 1/2\right] $} In this first case, we are going to use the following generalization of Guichal \cite{Guichal}, proved in \cite{cost-weak}:
\begin{Theorem} (\cite[Theorem 2.5]{cost-weak})\label{thm-guichal-gen}
Assume that $\lambda_k\geq 0$ for all $k \geq 1$ and that there is some $0 < \gamma _{\text{min}} \leq \gamma _{\text{max}}$ such that
\begin{equation}\label{gap-max}
	\forall k \geq 1, \quad \gamma _{\text{min}} \leq \sqrt{\lambda _{k+1}} - \sqrt{\lambda _{k}}  \leq \gamma _{\text{max}} .
\end{equation}
Then there exists $c_u>0$ independent of $T$ and $m$ such that any family $(\sigma _{m} )_{m\geq 1}$ which is biorthogonal to the family $(e^{\lambda _{k}t})_{k\geq 1}$ in $L^2(0,T)$ satisfies
\begin{equation*}
	\Vert \sigma _{m}  \Vert _{L^2(0,T)} ^2 \geq e^{-2\lambda _{ m} T} e^{\frac{1}{2\gamma _{\text{max}} ^2 T}} b(T,\gamma_{max},m),
\end{equation*}
with
\begin{equation}\label{b-T-m-max}
	b(T,\gamma_{max},m) = \frac{c_u ^2}{\C(m, \gamma _{\text{max}}, \lambda _1)^2 \, T } \left(\frac{1}{2\gamma _{\text{max}}^2 T}\right)^{2m} \frac{1}{(4\gamma _{max} ^2T+1)^2}
\end{equation}
and
\begin{equation*}
	\C(m, \gamma _{max}, \lambda _1) = m! \, 2 ^{m+ \left[\frac{2 \sqrt{\lambda _1}}{\gamma _{max}}\right] +1} \, \left(m+\left [\frac{2 \sqrt{\lambda _1}}{\gamma _{max}}\right] +1\right).
\end{equation*}
\end{Theorem}
When $\nu(\alpha,\mu) \in [0, \frac{1}{2}]$, using \eqref{eigen-gap1_unif}, we see that assumption \eqref{gap-max} is satisfied with
\begin{align*}
	\gamma _{\text{min}} := \frac{7\pi}{16}, \quad \text{ and } \quad \gamma _{\text{max}} := \pi.
\end{align*}

So, using Theorem \ref{thm-guichal-gen}, we obtain that any family $(\sigma _{m} )_{m\geq 1}$ which is biorthogonal to the family $(e^{\lambda _{k}t})_{k\geq 1}$ in $L^2(0,T)$ satisfies:
\begin{equation*}
	\Vert \sigma _{m}  \Vert _{L^2(0,T)} ^2 \geq e^{-2\lambda _{ m} T} e^{\frac{1}{2 \pi ^2 T}} b(T,\gamma_{max},m),
\end{equation*}
where $b(T,\gamma_{max},m)$ is given in \eqref{b-T-m-max}. Let us now apply this inequality for $m=1$. It implies 
\begin{equation}\label{mino-sigma1}
	\Vert \sigma _{1}  \Vert _{L^2(0,T)} ^2\geq e^{-2\lambda _{ 1} T} e^{\frac{1}{2 \pi ^2 T}} b(T,\gamma_{max},1).
\end{equation}
Next, we observe that, for $\nu \in [0,\frac 12]$ and $n=1$, \eqref{bessel_zero_bound} gives
\begin{align*}
	\frac{3\pi}{4}\leq \pi \left( \frac{3}{4} + \frac{\nu}{2}\right) \leq j_{\nu,1} \leq \pi \left( 1 + \frac{1}{4}\left( \nu - \frac{1}{2} \right)\right) \leq \pi.
\end{align*}
It follows that 
\begin{align*}
	\frac{9 \pi^2}{64}\leq \left( \frac{2-\alpha}{2} \right)^2 \left(\frac{3 \pi}{4}\right)^2 \leq \lambda_{1}  \leq  \left( \frac{2-\alpha}{2} \right)^2 \pi^2
 \leq \pi^2
\end{align*}
and 
\begin{align*}
	\lambda_{1} \leq  \C(1+\nu(\alpha, \mu))^2.
\end{align*}
In particular, using $\lambda_{1} \geq 9 \pi^2/64$, we obtain that 
\begin{align*}
	b(T,\gamma_{max},1) \geq \frac{\C}{T^3 (1+T)^2}.
\end{align*}
From \eqref{mino-sigma1}, we deduce  
\begin{align*}
	\Vert \sigma _{1}  \Vert _{L^2(0,T)} ^2 \geq \frac{\C}{T^3 (1+T)^2} e^{-2\lambda _{ 1} T} e^{\frac{1}{2\pi ^2 T}}.
\end{align*}
Hence,
\begin{align*}
	\Vert  \Phi'_1(1) H_1 \Vert _{L^2(0,T)} ^2 \geq \frac{\C}{T^3 (1+T)^2} e^{-2\lambda_{1}  T} e^{\frac{1}{2\pi ^2 T}}.
\end{align*}
From \eqref{eigen_deriv}, one has 
\begin{align*}
	|\Phi'_1(1)| = \frac{(2-\alpha)^{3/2} j_{\nu(\alpha, \mu),1}}{2} \leq \sqrt{2}  j_{\nu(\alpha, \mu),1}   \leq \sqrt{2} \pi.
\end{align*}
So we obtain that
\begin{align*}
	\C_{ctr-bd} \geq \frac{\C}{T^{3/2} (1+T)  } e^{-\lambda_{1} T} e^{\frac{\C}{ T}} .
\end{align*}
Finally, using the fact that $\lambda _{ 1} \leq \C (1+\nu(\alpha, \mu))^2$, we get 
\begin{align*}
	\C_{ctr-bd} \geq \frac{\C}{T^{3/2} (1+T)  } e^{- \C (1+\nu(\alpha, \mu))^2 T} e^{\frac{\C}{T}} \geq\ C  e^{\frac{C}{ T}}  e^{- \C (1+\nu(\alpha, \mu))^2 T}.
\end{align*}
Since $2-\alpha >1$, we have $\nu(\alpha,\mu) \leq  2 \sqrt{\mu(\alpha)-\mu}.$ Hence
\begin{align*}
	\C_{ctr-bd} \geq \C  e^{\frac{\C}{ T}}  e^{- \C \left[ 1+\sqrt{\mu(\alpha)-\mu}\right]^2 T},
\end{align*}
which gives the result. 

\subsubsection{Lower bound in the case  $\nu(\alpha, \mu) \in \left[1/2,+\infty\right)$}
In this case, one still could apply  Theorem \ref{thm-guichal-gen}. Indeed, assumption \eqref{gap-max} is satisfied with 
\begin{align*}
	\gamma _{min} := \frac{\pi}{2} , \quad \quad \gamma _{max} := \frac{2-\alpha}{2} [ j_{\nu(\alpha,\mu),2} - j_{\nu(\alpha,\mu),1}].
\end{align*}

Nevertheless, since $\gamma_{max} \to +\infty$ as $\mu \to -\infty$ (as mentionned in \eqref{gammamax-explose}), this would not give the best possible result. On the other hand, from Lemma \ref{asymp_gap_lemma}, one has 
\begin{equation*}
	\forall k \geq N_*, \quad \sqrt{\lambda _{k+1}} - \sqrt{\lambda _{ k}}  \leq \gamma _{\text{max}}^* 
\end{equation*}
with 
\begin{equation*}
	N_* := [\nu (\alpha, \mu)] +1 \quad \text{ and } \quad \gamma _{max} ^* := 2\pi.
\end{equation*}

In that context, when there is a \textit{bad global} upper gap $\gamma _{max}$, and a \textit{good} (much smaller) {\it asymptotic} upper gap $\gamma _{max} ^*$, it is interesting to use the following extension of Theorem \ref{thm-guichal-gen}:

\begin{Theorem} (\cite[Theorem 2.2]{CMV-biortho-general})\label{thm-guichal-gen*}
Assume that $\lambda_k\geq 0$ for all $k\geq 1$ and that there are $0<\gamma _{min} \leq \gamma _{\text{max}}^* \leq \gamma _{\text{max}}$ such that
\begin{equation*}
	\forall k \geq 1, \quad \gamma _{\text{min}} \leq \sqrt{\lambda _{k+1}} - \sqrt{\lambda _{k}}  \leq \gamma _{\text{max}} ,
\end{equation*}
and
\begin{equation*}
	\forall k \geq N_*, \quad \sqrt{\lambda _{k+1}} - \sqrt{\lambda _{k}}  \leq \gamma _{\text{max}}^* .
\end{equation*}
Then any family $(\sigma _{m} )_{m\geq 1}$ which is biorthogonal to the family $(e^{\lambda _{k}t})_{k\geq 1}$ in $L^2(0,T)$ satisfies
\begin{equation*}
	\Vert \sigma _{m}  \Vert _{L^2(0,T)} ^2 \geq e^{-2\lambda _{m} T} \, e^{\frac{2}{T (\gamma _{\text{max}}^*) ^2}} \, b^* (T,\gamma_{max},\gamma_{max} ^*, N_*, \lambda _1,m)^2,
\end{equation*}
where $b^*$ is rational in $T$ (and explictly given in \cite[Lemma 4.4]{CMV-biortho-general}).
\end{Theorem}

We are here in position to apply Theorem \ref{thm-guichal-gen*}. So we obtain that any family $(\sigma _{m} )_{m\geq 1}$ which is biorthogonal to the family $(e^{\lambda _{ k}t})_{k\geq 1}$ in $L^2(0,T)$ satisfies
\begin{align*}
	\Vert \sigma _{m}  \Vert _{L^2(0,T)}^2 \geq e^{-2\lambda _{ m} T} \, e^{\frac{2}{4\pi ^2 T}} \, b^* (T,\gamma_{max},\gamma_{max} ^*, N_*, \lambda _{1},m)^2,
\end{align*}
where, when $m \leq N_*$, $b^*$ (see Lemma 4.4 in \cite{CMV-biortho-general}) has the following explicit value of 
\begin{equation*}
	b^*(T,\gamma_{max},\gamma_{max}^*,N_*, \lambda _1,m) = \C^* \frac{\sqrt{1+T\lambda _1}}{\sqrt{T}} \, \frac{(T \, (\gamma _{max} ^*) ^2)^{K_*+K' _*+2}}{(1+ (T \, (\gamma _{max} ^*) ^2))^{N_*+K_*+K' _*+3}},
\end{equation*}
with 
\begin{align*}
	&K_* = \left[\frac{2\sqrt{\lambda _1}+(N_*+m)\gamma _{max}}{\gamma_{max}^*}\right] - N_* + 2,
	\\
	&K' _* =  \left[\frac{\gamma_{max}}{\gamma_{max}^*}(N_*-m)\right] -N_*+2,
	\\
	&\C^* = \frac{1}{(N_*+K_*+K'_*+3)!} \frac{c_u (\gamma _{max} ^*)^{2(N_*-1)} }{\C^{(+)} \C^{(-)}},
\end{align*}
and  with
\begin{align*}
	\C^{(+)} = \left(\frac{\gamma _{max}}{\gamma _{max} ^*}\right)^{N_* -1} \frac{\left(N_*+m+ \left[\frac{2\sqrt{\lambda _1}}{\gamma_{max}}\right]+1\right)!}{\left(m+ \left[\frac{2\sqrt{\lambda _1}}{\gamma_{max}}\right]+1\right)! \, \left(\left[\frac{2\sqrt{\lambda _1}+(N_*+m)\gamma_{max}}{\gamma_{max}^*}\right]+1\right)! \, \left(2m+ \left[\frac{2\sqrt{\lambda _1}}{\gamma_{max}}\right]+1\right)} 
\end{align*}
and
\begin{align*}
	\C^{(-)} =  \left(\frac{\gamma _{max}}{\gamma _{max} ^*}\right) ^{N_*-1} \, \frac{(m-1)! \, (N_*-m)! }{\left(1+ \left[\frac{\gamma_{max}}{\gamma_{max}^*}(N_*-m)\right]\right)!}.
\end{align*}

In the above expressions, we take $m=1$ and we only need to look at the behavior as $\mu \to -\infty$ i.e. $\nu (\alpha,\mu) \to +\infty$. This is possible to study (see \cite{CMV-the-cost-strong}) and one obtains
\begin{align*}
	b^* (T,\gamma_{max},\gamma_{max} ^*, N_*, \lambda _{1},1) \geq e^{-\C \nu (\alpha, \mu) ^{4/3} (\ln \nu (\alpha, \mu) + \ln \frac{1}{T})} 
\frac{\sqrt{1+T}}{\sqrt{T}}.
\end{align*}
Consequently,
\begin{align*}
	\Vert \sigma _{1}  \Vert _{L^2(0,T)} ^2 \geq \underline{b} (T,\alpha, \mu,1) ^2,
\end{align*}
with
\begin{equation*}
	\underline{b} (T,\alpha, \mu,1) := e^{-\lambda _{1} T} e^{\frac{1}{4 \pi ^2 T}}\, \frac{\sqrt{1+T}}{\sqrt{T}}e^{-\C \nu (\alpha, \mu) ^{4/3} (\ln \nu (\alpha, \mu) + \ln \frac{1}{T})}.
\end{equation*}
Hence,
\begin{align*}
	\Vert \Phi_1'(1) H_1 \Vert_{L^2(0,T)}  \geq \underline{b} (T,\alpha, \mu,1).
\end{align*}
This gives the following lower bound of the cost:
\begin{align*}
	\C_{ctr-bd} \geq \frac{1}{\vert \Phi_1'(1) \vert} \underline{b} (T,\alpha, \mu,1).
\end{align*}

From \eqref{bessel_zero_bound}, we have $j_{\nu(\alpha,\mu),1} \leq \C ( 1+\nu(\alpha,\mu)).$ We deduce that $\lambda_1 \leq \C (  1+\nu(\alpha,\mu)  )^2$ and, using also \eqref{eigen_deriv}, $\vert \Phi_1'(1) \vert \leq \C  (1+\nu(\alpha,\mu) )$. So, we get
\begin{align*}
	\C_{ctr-bd} &\geq \frac{\C}{1 +  \nu (\alpha, \mu)} \, e^{-\lambda _{1} T} \, e^{\frac{1}{4 \pi ^2 T}}\, e^{-\C \nu (\alpha, \mu) ^{4/3} (\ln \nu (\alpha, \mu) + \ln \frac{1}{T})}\, \frac{\sqrt{1+T}}{\sqrt{T}}
	\\
	&\geq \C  e^{-\C (1+ \nu(\alpha, \mu))^2 T}\,  e^{\frac{1}{4 \pi ^2 T}}\, e^{-\C \nu (\alpha, \mu) ^{4/3} (\ln \nu (\alpha, \mu) + \ln \frac{1}{T})}\, \frac{\sqrt{1+T}}{\sqrt{T}} 
	\\
	&\geq \C \,  e^{\frac{\C}{T}}e^{-\C \left[1+ \sqrt{\mu(\alpha)-\mu} \right]^2 T}\, e^{-\C \left[\sqrt{\mu(\alpha)-\mu} \right]^{4/3} 
\left(\ln [\sqrt{\mu(\alpha)-\mu} ] + \ln \frac{1}{T} \right)} .
\end{align*}

\section{Proof of Theorem \ref{control_thm0}}\label{sec-pf-control_thm0}

This Section is devoted to the proof of Theorem \ref{control_thm0} on the boundary controllability for \eqref{pb-x-0}. As for the case of a control acting at $x=1$, the proof will be organized into the following steps: 
\begin{itemize}
	\item[•] \textbf{Step 1.} The reduction to a moment problem.
	\item[•] \textbf{Step 2.} Formal solution.
	\item[•] \textbf{Step 3.} Existence and regularity of the control and upper bound of the cost.
	\item[•] \textbf{Step 4.} Lower bound of the cost.
\end{itemize}

Moreover, in what follows we are not going to present all the details of our computations, since they are in many aspects similar to the ones in the previous sections.

\subsection{Reduction to a moment problem} 

Once again, we expand the initial condition $u_0\in L^2(0,1)$ and the solution to \eqref{pb-x-0}  with respect to the basis of the eigenfunctions $(\Phi_k)_{k\geq 1}$: 
\begin{align*}
	u_0(x)=\sum_{k\geq 1}\rho_k^0\Phi_k(x),\;\;\; u(x,t)=\sum_{k\geq 1}\beta_k(t)\Phi_k(x).
\end{align*}
Therefore, the controllability condition $u(x,T)=0$ reads again as in \eqref{controllability_moment}

Besides, we notice again that the function $v_k(x,t):=\Phi_k(x)e^{\lambda_k(t-T)}$ solves the adjoint problem \eqref{heat_hardy_adj}. Hence,  combining \eqref{pb-x-0} and \eqref{heat_hardy_adj} we obtain
\begin{align*}
	0=& \int_Q \bigg[v_k\left(u_t-(x^\alpha u_x)_x-\frac{\mu}{x^{2-\alpha}}u\right)+u\left(v_{k,t}+(x^\alpha v_{k,x})_x+\frac{\mu}{x^{2-\alpha}}v_k\right)\bigg]\,dxdt
	\\
	=& \int_0^1 v_ku\,\Big|_0^T\,dx - \int_0^T x^\alpha u_xv_k\,\Big|_0^1\,dt + \int_0^T x^\alpha v_{k,x}u\,\Big|_0^1\,dt
	\\
	=& \int_0^1 v_k(x,T)u(x,T)\,dx - \int_0^1 v_k(x,0)u_0(x)\,dx - \int_0^T H(t)(x^{\alpha+\gamma}v_{k,x})(0,t)\,dt 
	\\
	=& \int_0^1 u(x,T)\Phi_k(x)\,dx - e^{-\lambda_k T}\int_0^1 u_0(x)\Phi_k(x)\,dx - e^{-\lambda_kT}r_k\int_0^T H(t)e^{\lambda_k t}\,dt
	\\
	=& \,\beta_k(T) - \rho_k^0e^{-\lambda_k T} - e^{-\lambda_k T}r_k\int_0^T H(t)e^{\lambda_kt}\,dt,
\end{align*}
where we have indicated
\begin{align}\label{rk}
	r_k:=\lim_{x\to 0^+} x^{\alpha+\gamma}\Phi'_k(x).
\end{align}
Then, from the controllability condition \eqref{controllability_moment} it follows that
\begin{align}\label{moment_cond0}
	\forall k\geq 1,\;\;\;-r_k\int_0^T H(t)e^{\lambda_kt}\,dt = \rho_k^0. 
\end{align}

On the other hand, since we are looking for a solution of the moment problem belonging to $H^1(0,T)$, instead of \eqref{moment_cond0} we would rather be interested in a condition involving the derivative of the function $H$. This condition can be obtained once again integrating by parts as follows
\begin{align*}
	\int_0^T H(t)e^{\lambda_kt}\,dt = \frac{1}{\lambda_k}H(t)e^{\lambda_k t}\,\bigg|_0^T-\frac{1}{\lambda_k}\int_0^T H'(t)e^{\lambda_k t}\,dt.
\end{align*}
Therefore, the derivative $H'(t)$ has to satisfy
\begin{align}\label{moment_cond_H1_0}
	\forall k\geq 1,\;\;\;\frac{r_k}{\lambda_k}\int_0^T H'(t)e^{\lambda_kt}\,dt = \rho_k^0 +\frac{r_k}{\lambda_k}\left(H(T)e^{\lambda_k T}-H(0)\right).
\end{align}

Also in this case, we will provide a solution to the above problem which satisfies $H(0)=H(T)=0$. To this end, we remark that the value of $r_k$ can be computed explicitly, starting from \eqref{rk} and using the definition of the Bessel function $J_{\nu(\alpha,\mu)}$. In particular, one can readily verify that

\begin{equation}\label{rk-expl}
	r_k= \frac{\mathcal B_1(\alpha,\mu)}{\Gamma(1+\nu(\alpha,\mu))}\jnk,
\end{equation}
with 
\begin{align*}
	\mathcal{B}_1(\alpha,\mu):=\frac{\sqrt{2-\alpha}}{|J_{\nu(\alpha,\mu)}'(\jnk)|}\left(\sqrt{\mu(\alpha)}+\sqrt{\mu(\alpha)-\mu}\right). 
\end{align*}

\subsection{Formal solution of the moment problem}

\subsubsection{Formal definition of the control $H$}
We present here the formal computations showing that the moment problem \eqref{moment_cond_H1_0} has a solution $H$. To define this function $H$, we will employ again the biorthogonal a sequence $(\sigma_k)_{k\geq 1}$ that we introduced before, and whose existence is guaranteed by the gap conditions \eqref{eigen-gap1} and \eqref{eigen-gap2} and by Theorem \ref{thm-biortho1-gen}

Now, let us define the function $H$ as follows:
\begin{align}\label{H_def-0}
	H(t):=\int_0^t K(s)\,ds, \quad \text{ with }\quad K(t):=\sum_{k\geq 1} \frac{\lambda_k}{r_k}\rho_k^0\sigma_k(t).
\end{align}

It is straightforward that, if $K\in L^2(0,T)$, then we have $H\in H^1(0,T)$ with $H'(t)=K(t)$ and $H(0)=0=H(T)$. Moreover,
\begin{align*}
	\frac{r_k}{\lambda_k}\int_0^T H'(t)e^{\lambda_k t}\,dt &= \frac{r_k}{\lambda_k}\int_0^T K(t)e^{\lambda_k t}\,dt = \frac{r_k}{\lambda_k}\int_0^T \left(\sum_{\ell\geq 1} \frac{\lambda_{\ell}}{r_{\ell}}\rho_{\ell}^0\sigma_{\ell}(t)\right)e^{\lambda_k t}\,dt
	\\
	&= \frac{r_k}{\lambda_k}\sum_{\ell\geq 1}\frac{\lambda_{\ell}}{r_{\ell}}\rho_{\ell}^0\int_0^T \sigma_{\ell}(t)e^{\lambda_k t}\,dt = \frac{r_k}{\lambda_k}\sum_{\ell\geq 1}\frac{\lambda_{\ell}}{r_{\ell}}\rho_{\ell}^0\delta_{k,\ell} = \rho_k^0,
\end{align*}
and the moment problem \eqref{moment_cond_H1_0} is formally satisfied.

\subsubsection{If regular, the control $H$ drives the solution from $u_0$ to zero}

We show here that the control $H$ that we introduced in \eqref{H_def-0} is able to drive the solution to \eqref{pb-x-0} from the initial state $u_0$ to zero in time $T$. To this end, let us remind the change of variables 
\begin{align*}
	v(x,t):=u(x,t)-x^\gamma\frac{p(x)}{p(0)}H(t),\;\;\; p(x):=1-x^q, \;\;\; q=2\sqrt{\mu(\alpha)-\mu},
\end{align*}
that transforms our original equation \eqref{pb-x-0} into
\begin{align*}
\renewcommand*{\arraystretch}{1.3}
	\left\{\begin{array}{ll}
		\displaystyle v_t-(x^\alpha v_x)_x-\frac{\mu}{x^{2-\alpha}}v = -x^\gamma\frac{p(x)}{p(0)}K(t), & (x,t)\in Q
		\\ 
		v(0,t)=v(1,t)=0, & t\in (0,T)
		\\ 
		v(x,0)=u_0(x), & x\in (0,1).
	\end{array}\right.
\end{align*} 
Now, for a fixed $\varepsilon>0$ we have
\begin{align*}
	\int_{\varepsilon}^T&\int_0^1 -x^\gamma\frac{p(x)}{p(0)}K(t)\Phi_k(x)e^{\lambda_k t}\,dxdt 
	\\
	&= \int_{\varepsilon}^T\int_0^1 \left(v_t-(x^\alpha v_x)_x-\frac{\mu}{x^{2-\alpha}}v\right)\Phi_k(x)e^{\lambda_k t}\,dxdt
	\\
	&= \int_0^1 v\Phi_ke^{\lambda_k t}\,\Big|_{\varepsilon}^T\,dx + \int_{\varepsilon}^T\int_0^1 v\left(-(x^\alpha\Phi_k')'-\frac{\mu}{x^{2-\alpha}}		\Phi_k-\lambda_k\Phi_k\right)e^{\lambda_k t}\,dxdt
	\\
	&=e^{\lambda_k T}\int_0^1 v(x,T)\Phi_k(x)\,dx - e^{\lambda_k \varepsilon}\int_0^1 v(x,\varepsilon)\Phi_k(x)\,dx.
\end{align*}
Hence, taking the limit for $\varepsilon\to 0^+$ we find
\begin{align*}
	\int_Q -x^\gamma\frac{p(x)}{p(0)}K(t)\Phi_k(x)e^{\lambda_k t}\,dxdt = e^{\lambda_k T}\int_0^1 v(x,T)\Phi_k(x)\,dx - \rho_k^0.
\end{align*}
From this last identity and \eqref{moment_cond_H1_0}, it immediately follows
\begin{align*}
	e^{\lambda_k T}\int_0^1 v(x,T)\Phi_k(x)\,dx &= \rho_k^0 + \left(\int_0^T K(t)e^{\lambda_k t}\,dt\right)\left(\int_0^1 -x^\gamma\frac{p(x)}{p(0)}\Phi_k(x)\,dx\right)
	\\
	&= \rho_k^0 +\frac{\lambda_k}{r_k}\rho_k^0\int_0^1 -x^\gamma\frac{p(x)}{p(0)}\Phi_k(x)\,dx.
\end{align*}
Moreover, 
\begin{align*}
	\int_0^1 & -x^\gamma\frac{p(x)}{p(0)}\Phi_k(x)\,dx 
	\\
	&= \frac{1}{\lambda_k}\int_0^1 -x^\gamma\frac{p(x)}{p(0)}\lambda_k\Phi_k(x)\,dx = \frac{1}{\lambda_k}\int_0^1x^\gamma\frac{p(x)}{p(0)}\left((x^\alpha\Phi_k'(x))'+\frac{\mu}{x^{2-\alpha}}\Phi_k(x)\right)\,dx
	\\
	&= \frac{1}{\lambda_k}\frac{p(x)}{p(0)}x^{\alpha+\gamma}\Phi_k'(x)\,\bigg|_0^1 - \frac{1}{\lambda_k}\int_0^1 \left(x^\gamma\frac{p(x)}{p(0)}\right)'x^{\alpha}\Phi_k'(x)\,dx + \frac{1}{\lambda_k}\int_0^1 x^\gamma\frac{p(x)}{p(0)}\frac{\mu}{x^{2-\alpha}}\Phi_k(x)\,dx 
	\\
	&=-\frac{r_k}{\lambda_k} - \frac{1}{\lambda_k}\left(x^\gamma\frac{p(x)}{p(0)}\right)'x^{\alpha}\Phi_k(x)\,\bigg|_0^1 + \frac{1}{\lambda_k}\int_0^1 \left[\left(x^{\alpha}\left(x^\gamma\frac{p(x)}{p(0)}\right)'\right)'+ \mu x^{\alpha+\gamma -2}\frac{p(x)}{p(0)}\right]\Phi_k(x)\,dx 
	\\
	&=-\frac{r_k}{\lambda_k} + \frac{1}{\lambda_k p(0)}\int_0^1 \Big[\left(x^{\alpha}\left(x^\gamma p(x)\right)'\right)'+ \mu x^{\alpha+\gamma-2}p(x)\Big]\Phi_k(x)\,dx = -\frac{r_k}{\lambda_k},
\end{align*}
since we already noticed that (see \eqref{Label})
\begin{align*}
	\big[x^{\alpha}\left(x^\gamma p\right)'\big]'(x)+ \frac{\mu}{x^{2-\alpha-\gamma}}p(x)=0. 
\end{align*} 
Hence, we get
\begin{align*}
	e^{\lambda_k T}\int_0^1 v(x,T)\Phi_k(x)\,dx = 0,
\end{align*}
which of course implies $v(x,T) = 0$ and, since $H(T)=0$, we can finally conclude that
\begin{align*}
	u(x,T)=v(x,T)+x^\gamma\frac{p(x)}{p(0)}H(T)=0.
\end{align*}

\subsection{Existence of the control, $H^1$ regularity and upper bound of the cost of controllability.}

We have to check that the control $H$ defined as in \eqref{H_def} belongs to $H^1(0,T)$ and to obtain the upper bound for the controllability cost. To this end, as we did before, we are going to prove instead that the function $K$ belongs to $L^2(0,T)$. 

In what follows, $\C_u$ denotes again a universal constant, independent of $T$, $\gamma_{max}$, $\gamma_{min}$ and $k$, which may change value even from line to line. From \eqref{H_def-0} we have 
\begin{align*}
	\norm{K}{L^2(0,T)} = \norm{\sum_{k\geq 1}\frac{\lambda_k}{r_k}\rho_k^0\sigma_k(t)}{L^2(0,T)} \leq \sum_{k\geq 1} |\,\rho_k^0|\left|\frac{\lambda_k}{r_k}\right|\norm{\sigma_k(t)}{L^2(0,T)}.
\end{align*}

Moreover, employing the expression \eqref{eigenv} of $\lambda_k$ and the explicit expression of $r_k$ given in \eqref{rk-expl}  we obtain
\begin{align*}
	\left|\frac{\lambda_k}{r_k}\right| &= \frac{(2-\alpha)^{\frac 32}\Gamma(1+\nu(\alpha,\mu))|J_{\nu(\alpha,\mu)}'(\jnk)|}{4\left(\sqrt{\mu(\alpha)}+\sqrt{\mu(\alpha)-\mu}\right)}\jnk \leq \C_u\frac{\Gamma(1+\nu(\alpha,\mu))}{\sqrt{\mu(\alpha)}+\sqrt{\mu(\alpha)-\mu}}\jnk, 
\end{align*}
since $0\leq\alpha<1$ and $|J_{\nu(\alpha,\mu)}'(\jnk)|\leq 1$ (see \cite[Formula 79]{cost-weak}). Therefore, we get
\begin{align*}
	& \norm{K}{L^2(0,T)} \leq \C_u\frac{\Gamma(1+\nu(\alpha,\mu))}{\sqrt{\mu(\alpha)}+\sqrt{\mu(\alpha)-\mu}} \norm{u_0}{L^2(0,1)}\left(\sum_{k\geq 1} \jnk^2\norm{\sigma_k(t)}{L^2(0,T)}^2\right)^{\frac 12}.
\end{align*}

From here, proceeding as in Section \ref{upp_bound_one_sec}, we can immediately conclude that $K\in L^2(0,T)$ and we have the following estimate
\begin{align*}
	\C _{ctr-bd} &\leq \C_u\frac{\Gamma(1+\nu(\alpha,\mu))}{\sqrt{\mu(\alpha)}+\sqrt{\mu(\alpha)-\mu}} e^{\frac{\C_u}{T}} \left[1+  \sqrt{\mu(\alpha)-\mu}\right] e^{- \C_u \left( 1+\sqrt{\mu(\alpha)-\mu} \right)^2 T}  .
\end{align*}

\subsection{Lower bound of the cost of controllability}\label{sec-pf-control_cost_thm0}
Fix $m\geq 1$ and consider the initial condition $u_0=\Phi_m$ in \eqref{pb-x-1}. Let $H_m$ be any control that drives the solution of the equation to zero in time $T$. Then, the moment condition \eqref{moment_cond0} yields
\begin{align*}
	-r_k\int_0^TH_m(t)e^{\lambda_kt}\,dt = \rho_k^0 = \int_0^1 u_0(x)\Phi_k(x)\,dx = \delta_{k,m}.
\end{align*}
Hence, 
\begin{align*}
	\forall\,k\geq 1, \;\;\;\int_0^T\Big(r_m H_m(t)\Big)e^{\lambda_kt}\,dt = r_m\frac{\delta_{k,m}}{r_k} = \begin{cases}
1, &\textrm{ if }k=m
	\\
	0, &\textrm{ if }k\neq m.
\end{cases}
\end{align*}

This means that the sequence $\big(r_m H_m(t)\big)_{\ell\geq 1}$ is biorthogonal to $(e^{\lambda_kt})_{k\geq 1}$ in $L^2(0,T)$. Now, as we did before, we choose $m=1$ and we distinguish between the two cases 
\begin{align*}
	\nu(\alpha,\mu)\in\left[0,\frac 12\right] \quad \text{ and } \quad \nu(\alpha,\mu)\in\left[\frac 12,+\infty\right).
\end{align*}
In the former one, employing \eqref{mino-sigma1} we have 
\begin{align*}
	\norm{r_1 H_1(t)}{L^2(0,T)}^2\geq e^{-2\lambda_1 T}e^{\frac{1}{2\gamma_{max}^2T}}b(T,\gamma_{\max},1),
\end{align*}
which yields 
\begin{align*}
	\norm{H_1(t)}{L^2(0,T)}^2& \geq \frac{1}{|r_1|}e^{-2\lambda_1 T}e^{\frac{1}{2\gamma_{max}^2T}}b(T,\gamma_{\max},1).
\end{align*}
Now, thanks to \eqref{rk-expl} we obtain
\begin{align*}
	\frac{1}{|r_1|} =  \frac{\Gamma(1+\nu(\alpha,\mu))}{\sqrt{\mu(\alpha)}+\sqrt{\mu(\alpha)-\mu}}\frac{|J_{\nu(\alpha,\mu)}'(j_{\nu(\alpha,\mu),1})|}{\sqrt{2-\alpha}}j_{\nu(\alpha,\mu),1}^{-1}.
\end{align*}

Moreover, since $0\leq\alpha<1$, employing \eqref{bessel_zero_bound} and the fact that $|J_\nu'(j_{\nu(\alpha,\mu),1})|\geq \C$ with $\C$ independent of $\mu$ (see \cite[Corollary 2]{cost-weak}), we also have
\begin{align*}
	\frac{1}{|r_1|} \geq  \frac{\C_u}{\sqrt{\mu(\alpha)}+\sqrt{\mu(\alpha)-\mu}}
\end{align*}
which yields
\begin{align*}
	\norm{H_1(t)}{L^2(0,T)}& \geq \frac{\C_u}{\sqrt{\mu(\alpha)}+\sqrt{\mu(\alpha)-\mu}}e^{-2\lambda_1 T}e^{\frac{1}{2\gamma_{max}^2T}}b(T,\gamma_{\max},1).
\end{align*}
Proceeding now as in the proof of Theorem \ref{control_thm1} it is easy to obtain our final estimate
\begin{align*}
	\C_{ctr-bd} \geq \frac{\C_u}{\sqrt{\mu(\alpha)}+\sqrt{\mu(\alpha)-\mu}}\frac{1}{T^4}e^{-\C_u(1-\alpha)^2T}e^{\frac{\C_u}{T}}.
\end{align*}
When $\nu(\alpha,\mu)\geq \frac 12$, instead, the lower bound reads as follows:
\begin{align*}
	\C_{ctr-bd} &\geq \frac{\C_u}{\sqrt{\mu(\alpha)}+\sqrt{\mu(\alpha)-\mu}}\,e^{\frac{\C}{T}}e^{-\C \left[1+ \sqrt{\mu(\alpha)-\mu} \right]^2 T}\, e^{-\C \left[\sqrt{\mu(\alpha)-\mu} \right]^{4/3} \left(\ln [\sqrt{\mu(\alpha)-\mu} ] + \ln \frac{1}{T} \right)}.
\end{align*}

The proof of this fact is totally analogous to what we already did in the proof of Theorem \ref{control_thm1} and we leave it to the reader.

\section{Final comments and open questions}\label{open_pb_sec}

In this paper, we have analyzed the controllability properties of a degenerate/singular parabolic equation on the space interval $(0,1)$. We have considered the two different situations of a boundary control acting at $x=1$ or $x=0$ (where the degeneracy/singularity occurs). In both cases, by means of the classical moment method, we have shown that the equation is null-controllable and we provided suitable estimates for the controllability cost. 

We present hereafter a non-exhaustive list of comments and open problems related to our work.

\begin{enumerate} 
	
	\item As a first thing, we recall that in the present work we are not treating the strongly degenerate case $1 \leq \alpha <2$. 
	\begin{itemize}
		\item When the control acts at $x=1$, we expect this case could be treated with a similar methodology (also using the ideas of \cite{CMV-the-cost-strong}). However, in order to keep the paper of a reasonable length, this case is not covered here.
		
		\item When the control acts at $x=0$, instead, this is an open question even in the purely degenerate case $\mu=0$ (\cite{CMV-the-cost-strong} deals only with controls in $x=1$). Indeed, in this case one encounters difficulty already at the level of the well-posedness of the equation, due to the need to find a suitable boundary condition.
	\end{itemize} 
	
	\item A second open problem is related to the obtaining of suitable Carleman estimates for boundary controllability.  This is not an easy task. Indeed, the usual weights introduced in previous works (\cite{biccari,sicon2008,memoire,cazacu,Ervedoza,Martinez,vanco1,Va-Zu-JFA}) for proving interior controllability are designed in such a way that all the boundary terms are greater or equal to zero, and can therefore be ignored. On the other hand, adapting these weights in order to keep the boundary terms and still be able to prove the Carleman is a quite cumbersome issue. Nevertheless, the interest in obtaining, if possible, a Carleman estimate for boundary controllability remains, and it is related to various further applications: 
	\begin{itemize}
		\item the treatment of equations with a nonlinear term;
		\item the possibility of considering general function $a(x)$ (such as in \cite{Martinez}) instead of $x^\alpha$ in the purely degenerate case (and even with a double degeneracy both at $x=0$ and $x=1$); 
		\item the possibility of studying problems for a purely singular operator with two singular points at $x=0$ and $x=1$;
		\item the case of a degenerate/singular operator with $\mu/x^\beta$ with $\beta \leq 2-\alpha$ (instead of $\mu/x^{2-\alpha}$). In this case (analyzed in \cite{vanco1} only limited to a locally distributed control), null controllability should be true for any $\mu$ but it cannot be studied with the present method.
	\end{itemize}
\end{enumerate}


\begin{thebibliography}{99}

\bibitem{fatiha}  F. Alabau-Boussouira, P. Cannarsa and  G. Fragnelli, {\it Carleman estimates for degenerate parabolic operators with applications to null controllability}, J. Evol. Equ., Vol. 6, Nr. 2 (2006), pp. 161-204.

\bibitem{biccari} U. Biccari, E. Zuazua, {\it Null controllability for a heat equation with a singular inverse-square potential involving the distance to the boundary function}, J. Differential Equations, Vol. 261, Nr. 5 (2016), pp. 2809-2853. 

\bibitem{biccari2} U. Biccari, {\it Boundary controllability for a one-dimensional heat equation with a singular inverse-square potential}, Math. Control Relat. Fields, Vol. 9, Nr. 1 (2019), pp. 191-219. 

\bibitem{non-div} P. Cannarsa, G. Fragnelli, D. Rocchetti, {\it Controllability results for a class of one-dimensional degenerate parabolic problems in nondivergence form},  J. Evol. Equ., Vol. 8, Nr. 4 (2008), pp. 583-616.

\bibitem{sicon2008} P. Cannarsa, P. Martinez, J. Vancostenoble, \textit{Carleman estimates for a class of degenerate parabolic operators}, SIAM J. Control Optim., Vol. 47, Nr. 1 (2008), pp. 1-19.

\bibitem{memoire} P. Cannarsa, P. Martinez, J. Vancostenoble, \textit{Global Carleman estimates for degenerate parabolic operators with applications}, Memoirs of the American Mathematical Society, Vol. 239, Nr. 1133 (2016).

\bibitem{cost-weak} P. Cannarsa, P. Martinez, J. Vancostenoble, \textit{The cost of controlling weakly degenerate parabolic equations by boundary controls}, Math. Control Relat. Fields, Vol. 17, Nr. 2 (2017), pp. 171-211. 

\bibitem{CMV-biortho-general} P. Cannarsa, P. Martinez, J. Vancostenoble, {\it Precise estimates for biorthogonal families under asymptotic gap conditions}, Discrete Contin. Dyn. Syst. Ser. S. (2019), pp. 555-590.

\bibitem{CMV-the-cost-strong} P. Cannarsa, P. Martinez, J. Vancostenoble, {\it The cost of controlling strongly degenerate parabolic equations}, to appear in ESAIM: Control Optim. Calc. Var., VOl. 26, Nr. 2 (2020).

\bibitem{cazacu} C. Cazacu, {\it Controllability of the heat equation with an inverse-square potential localized on the boundary}, SIAM J. Control Optim., Vol.52, Nr. 4 (2014), pp. 2055-2089.

\bibitem{Ervedoza}  S. Ervedoza, {\it Control and stabilization properties for a singular heat equation with an inverse-square potential},  Comm. Partial Differential Equations, Vol. 33, Nr. 10-12 (2008), pp. 1996–2019.

\bibitem{FR1} H. O. Fattorini, D. L. Russel, {\it Exact controllability theorems for linear parabolic equations in one space dimension}, Arch. Rat. Mech. Anal. Vol. 4 (1971), pp. 272-292.

\bibitem{FR2} H. O. Fattorini, D. L. Russel, {\it Uniform bounds on biorthogonal functions for real exponentials with an application to the control theory of parabolic equations}, Quart. Appl. Math. Vol. 32 (1974), pp. 45-69.

\bibitem{Fo-Sa} M. Fotouhi, L. Salimi, {\it Null controllability of degenerate/singular parabolic equations}, J. Dyn. Control Syst., Vol. 18, Nr. 4 (2012), pp. 573-602.

\bibitem{Gueye} M. Gueye, {\it Exact boundary controllability of 1-D parabolic and hyperbolic degenerate equations},  SIAM J. Control Optim. Vol. 52, Nr. 4 (2014), pp. 2037–2054.

\bibitem{Guichal} E.N. G\"uichal, {\it A lower bound of the norm of the control operator for the heat equation}, J. Math. Anal. Appl., Vol. 110, Nr. 2 (1985), pp. 519-527.

\bibitem{Maniar} A. Hajjaj, L. Maniar, J. Salhi, {\it Carleman estimates and null controllability of degenerate/singular parabolic systems}, Electron. J. Differential Equations, Vol. 2016, Nr. 292 (2016), pp. 1-25.

\bibitem{Kom-Lor} V. Komornik, P. Loreti, {\it Fourier series in control theory}, Springer, Berlin, 2005.

\bibitem{Lebedev} N.N. Lebedev, {\it Special functions and their applications}, Dover Publications, New York, 1972 

\bibitem{Lorch} L. Lorch, M.E. Muldoon, {\it Monotonic sequences related to zeros of Bessel functions}, Numer. Algor., Vol. 49, Nr. 1-4 (2008), pp. 221-233.

\bibitem{Martinez}  P. Martinez, J. Vancostenoble, \textit{Carleman estimates for one-dimensional degenerate heat equations}. J. Evol. Eq., Vol. 6, Nr. 2 (2006), pp. 325-362.

\bibitem{MV-cost-singular}  P. Martinez, J. Vancostenoble, \textit{The cost of boundary controllability for a parabolic equation with inverse square potential}, Evol. Equ. Control Theory, Vol. 8, Nr. 2 (2019), pp. 397-422.

\bibitem{QuWong} C. K. Qu, R. Wong,  {\it "Best possible'' upper and lower bounds for the zeros of the Bessel function $J_\nu(x)$}, Trans. Amer. Math. Soc., Vol. 351, Nr. 7 (1999), pp. 2833-2859. 
 
\bibitem{Seid-Avdon} T.I. Seidman, S.A. Avdonin, S.A. Ivanov, \textit{The ``window problem" for series of complex exponentials}, J. Fourier Anal. Appl., Vol. 6, Nr. 3 (2000), pp. 233-254.

\bibitem{Tucsnak} G. Tenenbaum, M. Tucsnak, {\it New blow-up rates for fast controls of Schr\"odinger and heat equations}, J. Differential Equations, Vol. 243, Nr. 1 (2007), pp. 70-100.

\bibitem{vanco1} J. Vancostenoble, {\it Improved Hardy-Poincar\'e inequalities and sharp Carleman estimates for degenerate/singular parabolic problems}, Discr. Cont. Dyn. Syst., Vol. 3, Nr. 3 (2011), pp. 761-790.

\bibitem{Va-Zu-JFA} J. Vancostenoble, E. Zuazua, {\it  Null controllability for the heat equation with singular inverse-square potentials},  J. Funct. Anal., Vol. 254, Nr. 7 (2008), pp. 1864-1902. 

\bibitem{Va-Zu} J. L. V\'azquez,   E. Zuazua, {\it The Hardy inequality and the asymptotic behaviour of the  heat equation with an inverse-square potential}, J. Funct. Anal., Vol. 173, Nr. 1  (2000), pp. 103-153.

\bibitem{Watson} G. N. Watson, {\it A treatise on the theory of Bessel functions}, second edition, Cambridge University Press, Cambridge, England, 1944. 
\end{thebibliography}
\end{document}